\documentclass[11pt]{article}

\usepackage{amsmath,amssymb,amsthm,mathtools,bm}
\usepackage{booktabs}
\usepackage{geometry}
\usepackage{enumitem}
\usepackage{float}
\usepackage{hyperref}
\usepackage{xcolor}
\usepackage{graphicx}

\newcommand{\R}{\mathbb{R}}

\newcommand{\A}{\mathcal{A}}
\newcommand{\E}{\mathcal{E}}
\newcommand{\CN}{\mathcal{C}_N}
\newcommand{\SN}{S_N}
\newcommand{\SNP}{S_N^+}
\newcommand{\dd}{\,\mathrm{d}}

\newcommand{\eps}{\varepsilon}
\newcommand{\sig}{\sigma}
\newcommand{\inner}[2]{\left\langle #1,#2\right\rangle}
\newcommand{\norm}[1]{\left\|#1\right\|}
\newcommand{\abs}[1]{\left|#1\right|}
\newcommand{\grad}{\nabla}
\newcommand{\half}{\frac12}

\theoremstyle{plain}
\newtheorem{theorem}{Theorem}[section]
\newtheorem{lemma}[theorem]{Lemma}
\newtheorem{proposition}[theorem]{Proposition}
\newtheorem{corollary}[theorem]{Corollary}

\theoremstyle{definition}
\newtheorem{definition}[theorem]{Definition}

\numberwithin{equation}{section}
\counterwithin{figure}{section}
\counterwithin{table}{section}

\title{\bf
A positive-conservative Fourier optimization method for ground states of Bose--Einstein condensates with higher-order interactions
}

\author{
Bo Lin,\thanks{Beijing Huairou Laboratory, Beijing 101400,  People's Republic of China} \quad
Xinran Ruan\thanks{(Corresponding author) School of Mathematical Sciences, Capital Normal University, Beijing 100048, People's Republic of China {\tt (xinran.ruan@cnu.edu.cn)}.}
}
\date{}

\begin{document}
\maketitle

\begin{abstract}
We develop a positive-conservative Fourier optimization (PCFO) method
for computing ground states of Bose--Einstein condensates with
higher-order interactions. The ground-state problem admits a convex
density formulation, but the singular behavior near zero density
poses difficulties for high-accuracy computation. We introduce
convex regularizations of the density energy and establish their
$\Gamma$-convergence on the computational domain as the regularization
parameters vanish. The regularized problem is discretized by a
Fourier pseudospectral method with mass conservation and nodal
positivity. 
For the resulting constrained optimization problem, we use an accelerated
first-order method for the main energy reduction, followed by a Newton
refinement to reduce the remaining optimality residual. 
Numerical experiments show spectral-type convergence
for regularized problems, quantify the effects of the
regularization parameters on accuracy and spatial resolution, and
demonstrate the effectiveness of the two-stage optimization method.
\end{abstract}


\medskip


\section{Introduction}

Bose--Einstein condensation, first realized experimentally in dilute
atomic gases in 1995 \cite{Anderson1995}, provides a fundamental
setting for the study of macroscopic quantum phenomena. 
In the dilute and low-temperature regime, Bose--Einstein condensates are commonly
described by the Gross--Pitaevskii equation \cite{BaoCaiReview}, in
which two-body interactions are modeled by a zero-range contact
potential. Finite-range and effective-range corrections arise beyond
this approximation
\cite{CollinMassignanPethick2007,FuWangGao2003} and may become more
relevant under strong confinement or near narrow Feshbach resonances
\cite{Chin2010,VekslerFishmanKetterle2014,ZinnerThogersen2009}.
One such correction leads to a modified Gross--Pitaevskii model with
a higher-order interaction involving the gradient of the density.

We consider the modified Gross--Pitaevskii model with higher-order
interactions (HOI), whose energy is
\begin{equation}
\label{eq:wave-energy}
E(\phi)
=
\int_{\R^d}
\left[
\half |\grad\phi|^2
+V(x)|\phi|^2
+\frac{\beta}{2}|\phi|^4
+\frac{\delta}{2}\bigl|\grad(|\phi|^2)\bigr|^2
\right]\dd x,
\end{equation}
subject to
\begin{equation}
\label{eq:wave-mass}
\norm{\phi}_{L^2}^2=1.
\end{equation}
Here $\phi$ is the macroscopic wave function of the condensate,
$V(x)\ge0$ is a confining potential, $\beta\ge0$ is the repulsive
contact-interaction strength, and $\delta>0$ is the higher-order
interaction strength.
The competition between the contact and
higher-order interactions can affect ground-state structures and
their asymptotic regimes \cite{BaoCaiRuan2019,RuanCaiBao2016}. Under
strong confinement, the higher-order interaction also modifies the
effective lower-dimensional mean-field models
\cite{RuanCaiBao2016}.

As shown in \cite{BaoCaiRuan2019}, a ground state may be chosen real
and nonnegative. This motivates working with the density
$\rho=|\phi|^2$, which gives the equivalent formulation
\begin{equation}
\label{eq:density-energy}
\E_0(\rho)
=
\int_{\R^d}
\left[
\frac{|\grad\rho|^2}{8\rho}
+V(x)\rho
+\frac{\beta}{2}\rho^2
+\frac{\delta}{2}|\grad\rho|^2
\right]\dd x,
\end{equation}
over
\begin{equation}
\label{eq:admissible-density}
\A
=
\left\{
\rho\in H^1(\R^d):
\rho\ge0,\quad
\int_{\R^d}\rho\,\dd x=1,\quad
\int_{\R^d}V\rho\,\dd x<\infty
\right\}.
\end{equation}
The quotient $|\grad\rho|^2/(8\rho)$ is interpreted in the extended
sense defined in Section~\ref{sec:density}. For $\beta\ge0$,
$\E_0$ is strictly convex on $\A$. This convexity is useful for
ground-state computation, although the kinetic term remains singular
at zero density.

Ground-state computation for the Gross--Pitaevskii equation has been
studied extensively. Early approaches include direct energy
minimization and normalized gradient-flow methods
\cite{BaoDu2004,BaoTang2003}, together with spectrally accurate
discretizations \cite{BaoChernLim2006}. For strongly nonlinear or
rotating condensates, Sobolev-gradient methods, preconditioned
spectral solvers, nonlinear conjugate-gradient methods, regularized
Newton methods, and Riemannian optimization have been developed
\cite{AntoineDuboscq2014,AntoineLevittTang2017,
DanailaKazemi2010,DanailaProtas2017,WuWenBao2017}. More recent work
includes normalized Sobolev-gradient flows and nonlinear eigenvalue
iterations with convergence analysis
\cite{AltmannHenningPeterseim2021,ChenLuLuZhang2024,
HenningPeterseim2020}. Fourier pseudospectral discretizations combined
with manifold optimization have also been applied to more general
spinor condensates \cite{TianCaiWuWen2020}.

For the HOI model, a normalized gradient-flow method with an
attractive--repulsive splitting was developed in \cite{Ruan2018}.
 Bao and Ruan \cite{BaoRuan2019} subsequently formulated the
ground-state problem as a convex minimization problem in the density
variable, using a shifted kinetic regularization, a second-order
finite-difference discretization, and an accelerated projected-gradient
method, with a square-root regularization of the potential considered
numerically.

To enable spectral spatial accuracy within the density formulation, we
further develop the framework of \cite{BaoRuan2019} by improving both
the spatial discretization and the optimization solver. The
second-order finite-difference discretization is replaced by a Fourier
pseudospectral method with exact mass conservation and nodal
positivity. The potential term is incorporated into the proximal step
of the accelerated first-order method, followed by a Newton refinement
to reduce the optimization error so that it does not limit the spatial
discretization accuracy. We also extend the shifted kinetic
regularization to a class of convexity-preserving denominators and
treat the potential regularization with an independent scale. The
resulting regularized energies are shown to $\Gamma$-converge on the
computational domain as the regularization parameters vanish. The
spatial accuracy, regularization effects, and solver performance are
then examined numerically.

The rest of the paper is organized as follows.
Section~\ref{sec:density} introduces the convex regularizations and
establishes their variational convergence.
Section~\ref{sec:ps} presents the Fourier pseudospectral discretization
and the projection preserving positivity and mass.
Section~\ref{sec:opt} develops the potential-proximal FISTA method and
the Newton refinement. Section~\ref{sec:num} reports the numerical
results, and Section~\ref{sec:summary} concludes the paper.
\section{Density formulation and convex regularization}
\label{sec:density}

In this section, we study the convex structure of the density energy,
introduce convex regularizations of the kinetic and potential terms,
and establish variational convergence of the regularized energies.

\subsection{Convex regularization of the kinetic and potential terms}

\paragraph{Kinetic regularization.}

We extend the density kinetic term to possible vacuum regions.
For $s\ge0$ and $q\in\R^d$, let
\begin{equation}
\label{eq:kinetic-integrand}
F_0(s,q)
=
\begin{cases}
\dfrac{|q|^2}{8s}, & s>0,\\[1ex]
0, & s=0,\ q=0,\\
+\infty, & s=0,\ q\ne0.
\end{cases}
\end{equation}
With this convention, the density energy in
\eqref{eq:density-energy} can be written as
\begin{equation}
\label{eq:density-energy-extended}
\E_0(\rho)
=
\int_{\R^d}
\left[
F_0(\rho,\nabla\rho)
+V\rho
+\frac{\beta}{2}\rho^2
+\frac{\delta}{2}|\nabla\rho|^2
\right]\dd x,
\qquad
\rho\in\A.
\end{equation}

The density energy has the following convexity property.

\begin{proposition}[Convexity of the density energy]
\label{prop:density-convexity}
Assume $\beta\ge0$ and $\delta>0$.
Then $\E_0$ is strictly convex on $\A$ and hence has at most one
minimizer.
\end{proposition}

\begin{proof}
The map $(s,q)\mapsto |q|^2/s$ is the perspective of the convex
quadratic function $q\mapsto |q|^2$, and
\eqref{eq:kinetic-integrand} is its lower-semicontinuous convex
extension to $s=0$. Hence the kinetic term is convex. The potential
term is linear, and the $\beta$ term is convex for $\beta\ge0$.

The $\delta$ term gives strict convexity on $\A$. Indeed, equality in
its convexity inequality for two densities in $\A$ implies that their
gradients coincide almost everywhere. Their difference is therefore
constant, and integrability on $\R^d$ implies that this constant is
zero.
\end{proof}

To remove the singularity at zero density while preserving the convex
structure, we regularize the kinetic denominator. For $s\ge0$ and
$q\in\R^d$, set
\begin{equation}
\label{eq:regularized-kinetic-integrand}
F_\eps(s,q)
=
\frac{|q|^2}{8r_\eps(s)},
\end{equation}
where $r_\eps$ is a positive regularization of the denominator.

\begin{definition}[Admissible denominator]
\label{def:admissible-regularization}
A family $\{r_\eps\}_{\eps>0}$ is called admissible if, for every
$\eps>0$,
\begin{equation}
\label{eq:regularization-admissibility-local}
r_\eps\in C^1([0,\infty)),
\qquad
r_\eps>0,
\qquad
r_\eps'\ge0,
\qquad
r_\eps \text{ is concave},
\end{equation}
and, for every $s\ge0$,
\begin{equation}
\label{eq:regularization-admissibility-limit}
s
\le
r_{\eps_1}(s)
\le
r_{\eps_2}(s)
\quad
\text{for }0<\eps_1\le\eps_2,
\qquad
r_\eps(s)\to s
\quad\text{as }\eps\to0^+.
\end{equation}
\end{definition}

The concavity of the denominator preserves the
convexity of the regularized kinetic term.

\begin{lemma}[Convexity of the regularized kinetic integrand]
\label{lem:regularized-kinetic-convexity}
Let $r_\eps$ be admissible in the sense of Definition \ref{def:admissible-regularization}. Then
\[
F_\eps(s,q)
=
\frac{|q|^2}{8r_\eps(s)}
\]
is jointly convex in $(s,q)\in[0,\infty)\times\R^d$.
\end{lemma}

\begin{proof}
Let $0\le\theta\le1$ and set
\[
s_\theta=\theta s_1+(1-\theta)s_2,
\qquad
q_\theta=\theta q_1+(1-\theta)q_2.
\]
By the concavity of $r_\eps$,
\[
r_\eps(s_\theta)
\ge
\theta r_\eps(s_1)+(1-\theta)r_\eps(s_2).
\]
Hence
\[
F_\eps(s_\theta,q_\theta)
\le
\frac{|q_\theta|^2}
{8[\theta r_\eps(s_1)+(1-\theta)r_\eps(s_2)]}.
\]
The weighted Cauchy--Schwarz inequality gives
\[
\frac{|q_\theta|^2}
{\theta r_\eps(s_1)+(1-\theta)r_\eps(s_2)}
\le
\theta\frac{|q_1|^2}{r_\eps(s_1)}
+
(1-\theta)\frac{|q_2|^2}{r_\eps(s_2)},
\]
which proves the result.
\end{proof} 

The simplest example is the shifted denominator introduced in
\cite{BaoRuan2019},
\begin{equation}
\label{eq:shift-denominator}
r_\eps(s)=s+\eps.
\end{equation}
For comparison, we also consider the family
\begin{equation}
\label{eq:piecewise-regularization-family}
r_{\eps,m}(s)
=
\begin{cases}
s+\eps-\dfrac{\eps}{m+1}
\left(1-\dfrac{s}{\eps}\right)^{m+1},
&0\le s\le\eps,\\[2ex]
s+\eps,
&s>\eps,
\end{cases}
\qquad m=1,2.
\end{equation}
These functions belong to $C^m([0,\infty))$ and satisfy
Definition~\ref{def:admissible-regularization}. 

\paragraph{Potential regularization.}

The kinetic regularization removes the singularity at $\rho=0$, but
the derivative of the linear potential term $V\rho$ at zero density
is $V$. In the discrete problem, this can activate the positivity
constraint and reduce the regularity relevant to the Fourier
approximation. We therefore regularize the potential term near zero
density by replacing $\rho$ with $p_\sig(\rho)$, where 
\begin{equation}
\label{eq:potential-regularizer}
p_\sig(s)
=
\sqrt{s^2+\sig^2}-\sig,
\qquad s\ge0.
\end{equation}
A square-root smoothing of this type was considered in
\cite{BaoRuan2019} as a numerical enhancement. Here $\sig$ is treated
as an independent regularization scale and is analyzed together with
the kinetic regularization. We have
\begin{equation}
\label{eq:potential-regularizer-derivatives}
p_\sig(0)=p_\sig'(0)=0,
\qquad
p_\sig''(s)
=
\frac{\sig^2}{(s^2+\sig^2)^{3/2}}
\ge0.
\end{equation}
Thus $Vp_\sig$ is convex for $V\ge0$. 
The condition
$p_\sig'(0)=0$ removes the nonzero derivative of the potential term
at zero density.
Moreover, $p_\sig(s)\to s$ as $\sig\to0^+$. For the convergence
analysis below, it is sufficient to take
$\sig=\sig(\eps)\to0$ as $\eps\to0^+$.

Combining the two regularizations, we obtain
\begin{equation}
\label{eq:joint-regularized-energy}
\E_{\eps,\sig}(\rho)
=
\int_{\R^d}
\left[
F_\eps(\rho,\nabla\rho)
+V(x)p_\sig(\rho)
+\frac{\beta}{2}\rho^2
+\frac{\delta}{2}|\nabla\rho|^2
\right]\dd x,
\qquad
\rho\in\A.
\end{equation}

The convex structure of the original density problem is retained by
the regularized energy.

\begin{proposition}[Convexity of the regularized energy]
\label{prop:joint-convexity}
Assume that $r_\eps$ is admissible in the sense of Definition \ref{def:admissible-regularization}, $V\ge0$, $\beta\ge0$,
$\delta>0$, and $\sig>0$.
Then $\E_{\eps,\sig}$ \eqref{eq:joint-regularized-energy} is strictly convex on $\A$ \eqref{eq:admissible-density}.
If $\beta>0$, it is strongly convex with respect to the $L^2$ norm.
\end{proposition}

\begin{proof}
By Lemma~\ref{lem:regularized-kinetic-convexity}, the kinetic
term is convex. The potential term is convex because
$p_\sig$ is convex and $V\ge0$, while the remaining terms are convex.
Strict convexity follows from the $\delta$ term as in
Proposition~\ref{prop:density-convexity}. If $\beta>0$, the term
$\frac{\beta}{2}\|\rho\|_{L^2}^2$ gives strong convexity in $L^2$.
\end{proof}

\subsection{\texorpdfstring{$\Gamma$}{Gamma}-convergence of the regularized energies}

To formulate the Fourier discretization on a bounded domain, we introduce
the periodic box
\[
\mathcal D_L=[-L,L)^d.
\]
On $\mathcal D_L$, we study the $\Gamma$-convergence of the regularized
energies as the regularization parameters vanish, with respect to the
strong $L^2(\mathcal D_L)$ topology \cite{Braides2002}.

Let $H^1_{\rm p}(\mathcal D_L)$ denote the periodic Sobolev space on
$\mathcal D_L$.
Define
\begin{equation}
\label{eq:admissible-density-box}
\A_L
=
\left\{
\rho\in H^1_{\rm p}(\mathcal D_L):
\rho\ge0\ {\rm a.e.},
\quad
\int_{\mathcal D_L}\rho\,\dd x=1
\right\}.
\end{equation}
For $\rho\in\A_L$, define
\begin{equation}
\label{eq:energies-DL}
\begin{aligned}
&\E_{\eps,\sig,L}(\rho)
=
\int_{\mathcal D_L}
\left[
F_\eps(\rho,\nabla\rho)
+Vp_\sig(\rho)
+\frac{\beta}{2}\rho^2
+\frac{\delta}{2}|\nabla\rho|^2
\right]\dd x, \\
&\E_{0,L}(\rho)
:=\E_{0,0,L}(\rho) =
\int_{\mathcal D_L}
\left[
F_0(\rho,\nabla\rho)
+V\rho
+\frac{\beta}{2}\rho^2
+\frac{\delta}{2}|\nabla\rho|^2
\right]\dd x. 
\end{aligned}
\end{equation}

The potential regularization is uniformly small on $\mathcal D_L$. Since
\begin{equation}
\label{eq:potential-energy-defect}
0\le s-p_\sig(s)\le\sig,
\qquad s\ge0,
\end{equation}
we have
\[
\left|
\int_{\mathcal D_L}
V\,[p_\sig(\rho)-\rho]\,\dd x
\right|
\le
\sig\int_{\mathcal D_L}V\,\dd x.
\]
Thus the potential contribution is a uniformly vanishing perturbation
as $\sig\to0$, and the main convergence argument can be carried out
for the kinetic regularization.

\begin{theorem}[$\Gamma$-convergence of the joint regularization]
\label{thm:gamma-joint}
Let $V\in C(\mathcal D_L)$ with $V\ge0$, $\beta\ge0$, and $\delta>0$.
Let $\{r_\eps\}$ be admissible in the sense of
Definition~\ref{def:admissible-regularization}, and let
$\sig(\eps)>0$ satisfy
$\sig(\eps)\to0$ as $\eps\to0^+$.
Extend $\E_{\eps,\sig(\eps),L}$ and $\E_{0,L}$ to $L^2(\mathcal D_L)$
by setting them equal to $+\infty$ outside $\A_L$.
Then
\[
\E_{\eps,\sig(\eps),L}
\xrightarrow{\Gamma}
\E_{0,L}
\qquad
\text{in the strong }L^2(\mathcal D_L)\text{ topology}.
\]
Moreover, the family
$\{\E_{\eps,\sig(\eps),L}\}_{\eps>0}$
is equicoercive in $L^2(\mathcal D_L)$.
\end{theorem}

\begin{proof}
We first consider the kinetic regularization alone and keep the
potential term unchanged. Define
\[
\widetilde\E_{\eps,L}(\rho)
=
\int_{\mathcal D_L}
\left[
F_\eps(\rho,\nabla\rho)
+V\rho
+\frac{\beta}{2}\rho^2
+\frac{\delta}{2}|\nabla\rho|^2
\right]\dd x.
\]

We first prove the liminf inequality. Let $\eps_n\to0$ and
$\rho_n\to\rho$ strongly in $L^2(\mathcal D_L)$. If
\[
\liminf_{n\to\infty}
\widetilde\E_{\eps_n,L}(\rho_n)=+\infty,
\]
there is nothing to prove. Otherwise, after passing to a subsequence,
we may assume that the energies are uniformly bounded. Since
$\delta>0$, the gradient term together with the fixed mass gives a
uniform $H^1(\mathcal D_L)$ bound. Hence, after extracting a further
subsequence,
\[
\rho_n\rightharpoonup\rho
\qquad
\text{weakly in }H^1(\mathcal D_L).
\]

Fix $\bar\eps>0$. For all sufficiently large $n$,
$\eps_n\le\bar\eps$. The ordering in
\eqref{eq:regularization-admissibility-limit} therefore gives
\[
F_{\eps_n}(s,q)
\ge
F_{\bar\eps}(s,q).
\]
For fixed $\bar\eps$, the integrand $F_{\bar\eps}$ is continuous and
jointly convex in $(s,q)$. The associated integral functional is
therefore weakly lower semicontinuous on $H^1(\mathcal D_L)$. Together with
the remaining terms, this yields
\[
\liminf_{n\to\infty}
\widetilde\E_{\eps_n,L}(\rho_n)
\ge
\widetilde\E_{\bar\eps,L}(\rho).
\]
Finally, let $\bar\eps\downarrow0$. Since
$r_{\bar\eps}(s)\downarrow s$, the nonnegative kinetic integrands
increase pointwise to $F_0$. The monotone convergence theorem then
gives
\[
\liminf_{n\to\infty}
\widetilde\E_{\eps_n,L}(\rho_n)
\ge
\E_{0,L}(\rho).
\]

For the recovery sequence, let $\rho\in\A_L$ satisfy
$\E_{0,L}(\rho)<\infty$ and take the constant sequence
$\rho_\eps=\rho$. The same monotone convergence argument gives
\[
\widetilde\E_{\eps,L}(\rho)
\longrightarrow
\E_{0,L}(\rho)
\qquad
\text{as }\eps\to0^+.
\]
Hence
\[
\widetilde\E_{\eps,L}
\xrightarrow{\Gamma}
\E_{0,L}
\qquad
\text{in strong }L^2(\mathcal D_L).
\]

It remains to restore the potential regularization.
By \eqref{eq:potential-energy-defect},
\[
\sup_{\rho\in\A_L}
\left|
\E_{\eps,\sig(\eps),L}(\rho)
-
\widetilde\E_{\eps,L}(\rho)
\right|
\le
\sig(\eps)\int_{\mathcal D_L}V\,\dd x
\longrightarrow0.
\]
Thus the potential smoothing is a uniformly vanishing perturbation
and does not change the $\Gamma$-limit. 

The equicoercivity follows from the $\delta$ term and the fixed mass,
which give a uniform $H^1(\mathcal D_L)$ bound, together with the compact
embedding $H^1(\mathcal D_L)\hookrightarrow L^2(\mathcal D_L)$.
\end{proof}

The $\Gamma$-convergence result gives the corresponding convergence
of minimum values and minimizers.

\begin{corollary}[Convergence of minimizers]
\label{cor:min-conv}
Under the assumptions of Theorem~\ref{thm:gamma-joint}, let
$\rho_{\eps,L}$ and $\rho_{0,L}$ denote the minimizers of
$\E_{\eps,\sig(\eps),L}$ and $\E_{0,L}$ on $\A_L$, respectively.
Then
\begin{equation}
\label{eq:minimizer-convergence}
\min_{\rho\in\A_L}\E_{\eps,\sig(\eps),L}(\rho)
\longrightarrow
\min_{\rho\in\A_L}\E_{0,L}(\rho),
\qquad
\rho_{\eps,L}
\longrightarrow
\rho_{0,L}
\quad\text{strongly in }L^2(\mathcal D_L).
\end{equation}
\end{corollary}

\begin{proof}
We first show that the minimizers exist and are unique.
Existence follows from the direct method. Indeed, the $\delta$ term
together with the fixed mass gives the required $H^1(\mathcal D_L)$
bound for a minimizing sequence, while weak lower semicontinuity and
compactness on the bounded periodic domain yield a minimizer in
$\A_L$.

For uniqueness, the convexity arguments of
Propositions~\ref{prop:density-convexity} and~\ref{prop:joint-convexity}
remain valid on $\mathcal D_L$. Since all terms in the energy are
convex, equality in the convexity inequality implies equality for the
$\delta$ term. Hence two minimizers must have the same gradient and
therefore differ by a constant. In the whole-space case this constant
is excluded by integrability, whereas on the periodic domain it is
excluded by the fixed-mass constraint. Thus the constant is zero and
the minimizer is unique.

Theorem~\ref{thm:gamma-joint} and equicoercivity then give the
convergence of the minimum values and the subsequential convergence of
$\rho_{\eps,L}$. Every convergent subsequence has a minimizer of
$\E_{0,L}$ as its limit. Since this minimizer is unique, all such
limits coincide with $\rho_{0,L}$, and therefore the whole family
converges.
\end{proof}

\section{Fourier pseudospectral discretization}
\label{sec:ps}

We discretize the regularized density problem on a periodic
computational domain by a Fourier pseudospectral method. We introduce
the discrete energy and its derivatives, together with the positive
mass-conservative projection used in the constrained optimization.

We present the one-dimensional formulation on $\mathcal D_L=[-L,L)$.
The multidimensional case follows by the standard tensor-product
extension. Let $N$ be even,
$x_j=-L+jh$, $j=0,\ldots,N-1$, and $h=2L/N$. For
$u,v\in\R^N$, define the discrete inner product and norm by
\begin{equation}
\label{eq:h-inner}
\inner{u}{v}_h
=
h\sum_{j=0}^{N-1}u_jv_j,
\qquad
\norm{u}_h^2=\inner{u}{u}_h.
\end{equation}

\subsection{Discrete energy, gradient and Hessian structure}

Let $D_N$ denote the real Fourier pseudospectral first-derivative
operator. For even $N$, we set the derivative of the Fourier mode
$k=N/2$ to zero, so that $D_N$ maps real grid vectors to real grid
vectors and satisfies
\begin{equation}
\label{eq:skew-adjoint}
D_N^T=-D_N
\end{equation}
with respect to the inner product \eqref{eq:h-inner}. The action of
$D_N$ is evaluated by FFTs.

For a grid density $\rho$, let $q=D_N\rho$. The discrete regularized
energy is
\begin{equation}
\label{eq:discrete-energy}
E_{\eps,\sig,N}(\rho)
=
h\sum_{j=0}^{N-1}
\left[
\frac{q_j^2}{8r_\eps(\rho_j)}
+V_jp_\sig(\rho_j)
+\frac{\beta}{2}\rho_j^2
+\frac{\delta}{2}q_j^2
\right],
\end{equation}
subject to the discrete positivity and mass constraints
\begin{equation}
\label{eq:discrete-feasible}
\CN
=
\left\{
\rho\in\R^N:
\rho_j\ge0,
\quad
h\sum_{j=0}^{N-1}\rho_j=1
\right\}.
\end{equation}
Proposition~\ref{prop:discrete-convex} shows that the discrete energy
\eqref{eq:discrete-energy} retains the convexity of the regularized
problem.

\begin{proposition}[Discrete convexity]
\label{prop:discrete-convex}
Let $r_\eps$ be admissible in the sense of
Definition~\ref{def:admissible-regularization}. Let $\sig>0$ and assume
$V_j\ge0$, $\beta\ge0$, and $\delta>0$.
Then $E_{\eps,\sig,N}$ is convex on $\CN$ and admits a minimizer.
If $\beta>0$, it is strongly convex with respect to
$\norm{\cdot}_h$, and the minimizer is unique.
\end{proposition}

\begin{proof}
By Lemma \ref{lem:regularized-kinetic-convexity} and the linearity of
\[
\rho\mapsto \bigl(\rho_j,(D_N\rho)_j\bigr),
\]
the discrete kinetic term is convex. The remaining terms inherit
the convexity properties established in Proposition \ref{prop:joint-convexity}.
Since $\CN$ is compact and $E_{\eps,\sig,N}$ is continuous, a
minimizer exists. If $\beta>0$, the term
$\frac{\beta}{2}\|\rho\|_h^2$ gives strong convexity and hence
uniqueness.
\end{proof}

Differentiating \eqref{eq:discrete-energy} gives the discrete gradient.
Let $q=D_N\rho$. With all nonlinear operations understood
pointwise, the gradient with respect to the inner product
\eqref{eq:h-inner} is characterized by
\begin{equation}
\label{eq:discrete-gradient}
\begin{aligned}
&G_{\eps,\sig,N}(\rho)
=
\frac14D_N^T
\left(
\frac{q}{r_\eps(\rho)}
\right)
-\frac18
\frac{q^2r_\eps'(\rho)}
     {r_\eps(\rho)^2}
+V\,p_\sig'(\rho)
+\beta\rho
+\delta D_N^TD_N\rho,\\
&\left.
\frac{\dd}{\dd t}
E_{\eps,\sig,N}(\rho+t\eta)
\right|_{t=0}
=
\inner{G_{\eps,\sig,N}(\rho)}{\eta}_h .
\end{aligned}
\end{equation}

At points where $r_\eps$ is twice differentiable, define
\begin{equation}
\label{eq:kinetic-hessian-factorization}
\begin{aligned}
B
&=
D_N-
\operatorname{diag}
\left(
\frac{q\,r_\eps'(\rho)}
     {r_\eps(\rho)}
\right),
&
W
&=
\operatorname{diag}
\left(
\frac{1}{4r_\eps(\rho)}
\right).
\end{aligned}
\end{equation}
The Hessian contribution of the regularized kinetic term admits the
factorization
\begin{equation}
\label{eq:kinetic-hessian}
H_{\rm kin}
=
B^TWB
+
\operatorname{diag}
\left(
-\frac{q^2r_\eps''(\rho)}
       {8r_\eps(\rho)^2}
\right).
\end{equation}
Since $r_\eps>0$ and $r_\eps''\le0$ wherever the second derivative
exists, both terms in \eqref{eq:kinetic-hessian} are positive
semidefinite. The Hessian of the full discrete energy therefore has
the decomposition
\begin{equation}
\label{eq:full-hessian-spd-form}
H
=
H_{\rm kin}
+\delta D_N^TD_N
+\beta I
+\operatorname{diag}\!\bigl(Vp_\sig''(\rho)\bigr).
\end{equation}
Hence $H$ is symmetric positive semidefinite. If $\beta>0$, it is
positive definite. This structure will be used in the Newton refinement in
Section~\ref{subsec:Newton-PCG}.

For the shifted denominator $r_\eps(s)=s+\eps$, which is used in the
Newton refinement below, $r_\eps'=1$ and $r_\eps''=0$. The diagonal
correction in \eqref{eq:kinetic-hessian} then vanishes. The
corresponding matrix-free Hessian action is recorded in
Appendix~\ref{app:newton-implementation}.

\subsection{Positive mass-conservative projection}

The constrained optimization requires a projection that preserves
nonnegativity and mass. Let $\SN$ denote the trigonometric
interpolation space associated with the Fourier grid, and define
\begin{equation}
\label{eq:SNplus}
\SNP
=
\left\{
v_N\in\SN:
v_N(x_j)\ge0,\quad j=0,\ldots,N-1
\right\},
\end{equation}
where, for $v_N\in\SN$, we use the same discrete norm notation
$\norm{v_N}_h^2
=
h\sum_{j=0}^{N-1}|v_N(x_j)|^2$.
For $z_N\in\SN$ satisfying $\int_{\mathcal D_L}z_N\,\dd x \ge 0$, consider the positive mass-conservative
pseudospectral projection \cite{lin2024sinum}
\begin{equation}
\label{eq:spectral-positive-proj}
\operatorname*{argmin}_{g_N\in\SNP}
\frac12\norm{g_N-z_N}_h^2
\quad
\text{subject to}
\quad
\int_{\mathcal D_L}g_N\,\dd x
=
\int_{\mathcal D_L}z_N\,\dd x .
\end{equation}
The following proposition gives its nodal form.

\begin{proposition}[Nodal form of the positive mass-conservative projection]
\label{prop:proj-equiv}
Let $z=(z_j)_{j=0}^{N-1}\in\R^N$, where $z_j=z_N(x_j)$.
Then \eqref{eq:spectral-positive-proj} is equivalent to
\[
\operatorname*{argmin}_{y\in\R^N}
\frac12\norm{y-z}_h^2
\]
subject to
\[
y_j\ge0,
\qquad
h\sum_{j=0}^{N-1}y_j
=
\int_{\mathcal D_L}z_N\,\dd x .
\]
\end{proposition}

\begin{proof}
By the definition of the discrete norm,
\[
\norm{g_N-z_N}_h^2
=
h\sum_{j=0}^{N-1}
|g_N(x_j)-z_N(x_j)|^2.
\]
Moreover, the trapezoidal rule is exact for trigonometric
polynomials, so
\[
\int_{\mathcal D_L}g_N\,\dd x
=
h\sum_{j=0}^{N-1}g_N(x_j).
\]
Hence the objective and constraints in
\eqref{eq:spectral-positive-proj} coincide with those of the nodal
problem.
\end{proof}

Proposition~\ref{prop:proj-equiv} shows that the pseudospectral
projection can be carried out entirely in terms of the nodal values.
For the unit-mass constraint used in the optimization, we therefore
define the nodal projection onto $\CN$ by
\begin{equation}
\label{eq:nodal-projection}
P_{\CN}(z)
=
\operatorname*{argmin}_{y\in\CN}
\frac12\norm{y-z}_h^2,
\qquad z\in\R^N.
\end{equation}
It is given componentwise by
\begin{equation}
\label{eq:simplex-form}
\bigl[P_{\CN}(z)\bigr]_j
=
\max\{z_j-\lambda,0\},
\qquad
h\sum_{j=0}^{N-1}
\bigl[P_{\CN}(z)\bigr]_j
=
1,
\end{equation}
where $\lambda$ is uniquely determined by the mass constraint.

\section{Two-stage optimization method}
\label{sec:opt}

In this section, we solve the discrete convex minimization problem
using a two-stage strategy. We first apply an accelerated proximal
method to reduce the energy efficiently, and then refine the resulting
state through the KKT system to reduce the stationarity residual.

The discrete ground-state problem is
\begin{equation}
\label{eq:discrete-min}
\min_{\rho\in\CN}E_{\eps,\sig,N}(\rho).
\end{equation}
Since the objective and $\CN$ are convex, every KKT point is a global
minimizer.

To assess stationarity in both stages, for a fixed $\tau_r>0$ define
the projected-gradient residual
\begin{equation}
\label{eq:pg-map}
\mathcal G_{\tau_r}(\rho)
=
\frac{1}{\tau_r}
\left[
\rho
-
P_{\CN}
\left(
\rho-\tau_rG_{\eps,\sig,N}(\rho)
\right)
\right].
\end{equation}
For any feasible $\rho$, the optimality condition for
\eqref{eq:discrete-min} is equivalent to
$\mathcal G_{\tau_r}(\rho)=0$. We use
$\norm{\mathcal G_{\tau_r}(\rho)}_h$ as the stationarity measure.
For the numerical experiments, we take $\tau_r=1$.

\subsection{Potential-proximal FISTA}
\label{subsec:potential_FISTA}

FISTA is an accelerated proximal-gradient method for composite convex
optimization \cite{BeckTeboulle2009}. We apply it to the discrete
problem \eqref{eq:discrete-min} by splitting the energy into a smooth part and a proximal
part.
As shown in  \eqref{eq:potential-regularizer-derivatives}, the potential regularization has curvature $p_\sig''(0)=1/\sig$, which
becomes large as $\sig$ decreases. 
We therefore include the regularized
potential together with the positivity and mass constraints in the
proximal part. Thus we use the splitting
\begin{equation}
\label{eq:f-plus-g}
E_{\eps,\sig,N}(\rho)
=
f_{\eps,N}(\rho)+g_{\sig,N}(\rho),
\end{equation}
where
\begin{equation}
\label{eq:smooth-part}
f_{\eps,N}(\rho)
=
h\sum_j
\left[
\frac{(D_N\rho)_j^2}{8r_\eps(\rho_j)}
+\frac{\beta}{2}\rho_j^2
+\frac{\delta}{2}(D_N\rho)_j^2
\right]
\end{equation}
is the smooth part, and
\begin{equation}
\label{eq:prox-part}
g_{\sig,N}(\rho)
=
h\sum_j V_jp_\sig(\rho_j)
+
I_{\CN}(\rho)
\end{equation}
is the proximal part. Here $I_{\CN}$ denotes the indicator function
of $\CN$.
The gradient of $f_{\eps,N}$ with respect to the discrete inner
product \eqref{eq:h-inner} is
\begin{equation}
\label{eq:smooth-gradient}
\nabla_h f_{\eps,N}(\rho)
=
\frac14D_N^T
\left(
\frac{D_N\rho}{r_\eps(\rho)}
\right)
-\frac18
\frac{(D_N\rho)^2r_\eps'(\rho)}
     {r_\eps(\rho)^2}
+\beta\rho
+\delta D_N^TD_N\rho,
\end{equation}
where the nonlinear operations are understood pointwise. 
For the regularizations  \eqref{eq:shift-denominator}--\eqref{eq:piecewise-regularization-family}, \(r_\eps'\) is Lipschitz continuous, and hence \(\nabla_h f_{\eps,N}\) is Lipschitz continuous on \(\mathcal C_N\). 

We use the Chambolle--Dossal variant of FISTA with backtracking
\cite{BeckTeboulle2009, ChambolleDossal2015}.
Given two consecutive iterates $\rho^{k-1},\rho^k\in\CN$, one FISTA
iteration consists of an extrapolation to $y^k$, a forward step to
$z^k$, and a proximal update to $\rho^{k+1}\in\CN$.

The extrapolated state is
\begin{equation}
\label{eq:fista-extrapolation}
y^k
=
\rho^k+\theta_k(\rho^k-\rho^{k-1}),
\end{equation}
where $\theta_k$ is the Chambolle--Dossal extrapolation parameter.
In the computations,
\(
\theta_k=(k-1)/(k+4).
\)
The extrapolation preserves the mass constraint. Since
$f_{\eps,N}$ is defined for nonnegative densities, the extrapolation
is accepted only if $y^k\ge0$. Otherwise, it is restarted by setting
\[
y^k=\rho^k.
\]

For a trial step size $\tau_k>0$, the forward step is
\begin{equation}
\label{eq:fista-forward}
z^k
=
y^k-\tau_k\nabla_h f_{\eps,N}(y^k).
\end{equation}
The proximal update is
\begin{equation}
\label{eq:potential-prox}
\rho^{k+1}
=
\operatorname*{argmin}_{\rho\in\CN}
\left\{
\frac12\norm{\rho-z^k}_h^2
+
\tau_k h\sum_jV_jp_\sig(\rho_j)
\right\}.
\end{equation}
The positivity and mass constraints are therefore enforced in the
proximal step.

The trial step size is accepted when the standard majorization
condition
\begin{equation}
\label{eq:fista-backtracking}
\begin{aligned}
f_{\eps,N}(\rho^{k+1})
\le\;&
f_{\eps,N}(y^k)
+
\inner{\nabla_h f_{\eps,N}(y^k)}
       {\rho^{k+1}-y^k}_h 
+
\frac{1}{2\tau_k}
\norm{\rho^{k+1}-y^k}_h^2
\end{aligned}
\end{equation}
is satisfied. Otherwise, $\tau_k$ is reduced and the
forward--proximal step is repeated. We take $\tau_1=1$, reuse the
previously accepted step size at each subsequent iteration, and halve
it during backtracking.

The proximal problem \eqref{eq:potential-prox} can be solved through
a single scalar mass multiplier. For a fixed $\lambda$, each positive
component of $\rho^{k+1}$ satisfies
\begin{equation}
\label{eq:potential-prox-node}
s-z_j^k
+\tau_kV_jp_\sig'(s)
+\lambda
=
0,
\qquad s>0.
\end{equation}
The left-hand side is strictly increasing in $s$, since
\(
1+\tau_kV_jp_\sig''(s)>0.
\)
Hence the positive solution, when it exists, is unique. The $j$th
component is zero when $z_j^k\le\lambda$ and otherwise is given by
the unique positive solution of
\eqref{eq:potential-prox-node}. The multiplier $\lambda$ is determined
by the mass constraint
\[
h\sum_j\rho_j^{k+1}=1.
\]
Thus the proximal update consists of independent scalar nodal equations coupled only through the mass constraint. The nodal equations are solved by Newton's method for a given \(\lambda\), while \(\lambda\) is determined from the mass constraint by a one-dimensional root search.

For the unregularized linear potential, $p(s)=s$,
\eqref{eq:potential-prox} reduces to
\[
\rho^{k+1}
=
P_{\CN}(z^k-\tau_kV).
\]
Hence the potential-proximal update is a nonlinear extension of the
nodal projection \eqref{eq:nodal-projection}.

Once the energy has stabilized and the projected residual \eqref{eq:pg-map} is
sufficiently small, the iterate is passed to the Newton refinement.

\subsection{Newton refinement}
\label{subsec:Newton-PCG}

After the FISTA stage, we use a Newton refinement to further reduce
the stationarity residual \eqref{eq:pg-map}. The corresponding KKT conditions of
\eqref{eq:discrete-min} are 
\begin{equation}
\label{eq:kkt}
\begin{aligned}
& G_{\eps,\sig,N}(\rho)+\lambda\bm1-\mu=0,
\qquad
h\bm1^T\rho=1,\\
& \rho\ge0,\qquad
\mu\ge0,\qquad
\rho_j\mu_j=0.
\end{aligned}
\end{equation}
When the solution is interior, the inequality multiplier vanishes and
\eqref{eq:kkt} reduces to
\begin{equation}
\label{eq:interior-kkt}
G_{\eps,\sig,N}(\rho)+\lambda\bm1=0,
\qquad
h\bm1^T\rho=1.
\end{equation}

Starting from the FISTA output, each interior Newton iteration
computes a correction from the linearized KKT system and applies a
damped update. In the strongly convex regime $\beta>0$, the Hessian
is symmetric positive definite by
\eqref{eq:full-hessian-spd-form}. For the current iterate
$(\rho^k,\lambda^k)$, define
\[
r^k
=
G_{\eps,\sig,N}(\rho^k)+\lambda^k\bm1,
\qquad
r_m^k
=
h\bm1^T\rho^k-1
\]
as the stationarity and mass residuals, respectively.
Let $H^k=H(\rho^k)$ denote the Hessian of
$E_{\eps,\sig,N}$ at $\rho^k$ \eqref{eq:full-hessian-spd-form}.
The Newton correction
$(d^k,d_\lambda^k)$ satisfies
\[
H^kd^k+d_\lambda^k\bm1=-r^k,
\qquad
h\bm1^Td^k=-r_m^k.
\]

Rather than solving the indefinite KKT system directly, we use a
Schur complement. Let
\[
H^kz^k=r^k,
\qquad
H^kw^k=\bm1.
\]
Then
\begin{equation}
\label{eq:schur-two-solves}
d_\lambda^k
=
\frac{r_m^k-h\bm1^Tz^k}
     {h\bm1^Tw^k},
\qquad
d^k
=
-z^k-d_\lambda^kw^k.
\end{equation}
Thus each interior Newton step requires two linear systems with the
same symmetric positive definite Hessian. We solve these systems by
the preconditioned conjugate-gradient (PCG) method.

The correction is applied with a damping factor
$\alpha_k\in(0,1]$:
\[
\rho^{k+1}
=
\rho^k+\alpha_kd^k,
\qquad
\lambda^{k+1}
=
\lambda^k+\alpha_kd_\lambda^k.
\]
The damping factor is chosen to preserve positivity, together with a
residual line search to globalize the Newton iteration.
If active constraints are
encountered, a standard primal--dual active-set treatment can be used 
\cite{Ulbrich2011}. 
In the numerical experiments, the Newton refinement is terminated when the projected-gradient residual reaches the prescribed tolerance, with a maximum of 20 Newton iterations.

For the PCG solves, we use a sparse variable-coefficient
finite-difference approximation of the Fourier pseudospectral
Hessian as a symmetric positive definite preconditioner. Its
construction is given in
Appendix~\ref{app:newton-implementation}.

\section{Numerical experiments}
\label{sec:num}

In this section, we examine the spatial accuracy, regularization
effects, and computational performance of the positive-conservative
Fourier optimization (PCFO) method, and present representative
two-dimensional examples.

Unless otherwise stated, we take $\beta=\delta=10$, impose the
unit-mass constraint, use the shifted denominator
$r_\eps(\rho)=\rho+\eps$, and adopt the potential regularization
\eqref{eq:potential-regularizer}.
\subsection{Spatial accuracy}
\label{subsec:num-mesh}

We first examine the spatial accuracy of the Fourier discretization
with fixed regularization parameters. 
To separate the spatial error from
the regularization error, we work on the periodic interval
$\mathcal D_L=[-L,L)$ with $L=32$ and fix
$\eps=\sig=10^{-2}$. 

We take the harmonic potential $V_H(x)=x^2/2$. Its periodic extension
is continuous but not smoothly periodic, since
$V_H'(-L)\ne V_H'(L)$, which limits the Fourier accuracy at high
resolution. For the spatial-convergence test, we therefore use
\begin{equation}
\label{eq:periodic-harmonic}
V_L^{\rm p}(x)
=
\bigl(1-\chi_L(x)\bigr)\frac{x^2}{2}
+
\chi_L(x)\frac{L^2}{2},
\end{equation}
where
\[
\chi_L(x)
=
S\left(\frac{|x|-0.75L}{0.15L}\right),
\qquad
S(t)
=
\begin{cases}
0, & t\le0,\\
\displaystyle
\frac{\exp(-1/t)}{\exp(-1/t)+\exp(-1/(1-t))},
& 0<t<1,\\
1, & t\ge1.
\end{cases}
\]
Thus $V_L^{\rm p}$ agrees with the harmonic potential in the
interior and is constant near the periodic boundary.

Let $\rho_N$ and $\rho_{\rm ref}$ denote the trigonometric
interpolants of the numerical solutions on the $N$-point grid and
the finest comparison grid, respectively. We define
\begin{equation}
\label{eq:spatial-errors}
\begin{aligned}
e_{\rho,2}(N)
&=
\norm{\rho_N-\rho_{\rm ref}}_{L^2(\mathcal D_L)},
\qquad
e_{\rho,\infty}(N)
&=
\max_{0\le j<N}
\left|
\rho_N(x_j)-\rho_{\rm ref}(x_j)
\right|.
\end{aligned}
\end{equation}
The energy error is
\begin{equation}
\label{eq:spatial-energy-error}
e_E(N)
=
\abs{
E_{\eps,\sig,N}(\rho_N)
-
E_{\eps,\sig,N_{\rm ref}}(\rho_{\rm ref})
}.
\end{equation}

\begin{table}[H]
\centering
\small
\caption{
Fourier mesh refinement for $\eps=\sig=10^{-2}$ using the harmonic
trap with a $C^\infty$ periodic continuation \eqref{eq:periodic-harmonic}.
The comparison state is computed with $N_{\rm ref}=8192$.
}
\label{tab:mesh}
\begin{tabular}{ccccccc}
\toprule
$N$
&
$e_{\rho,2}$
&
rate
&
$e_{\rho,\infty}$
&
rate
&
$e_E$
&
rate
\\
\midrule
32
& $2.519\times10^{-2}$
& --
& $7.577\times10^{-3}$
& --
& $5.831\times10^{-3}$
& --
\\
64
& $3.394\times10^{-3}$
& $2.89$
& $1.766\times10^{-3}$
& $2.10$
& $2.077\times10^{-3}$
& $1.49$
\\
128
& $1.190\times10^{-4}$
& $4.83$
& $9.460\times10^{-5}$
& $4.22$
& $5.597\times10^{-6}$
& $8.54$
\\
256
& $6.227\times10^{-7}$
& $7.58$
& $2.841\times10^{-7}$
& $8.38$
& $3.770\times10^{-9}$
& $10.54$
\\
512
& $6.415\times10^{-11}$
& $13.24$
& $3.402\times10^{-11}$
& $13.03$
& $4.174\times10^{-14}$
& 16.46
\\
\bottomrule
\end{tabular}
\end{table}

Table~\ref{tab:mesh} reports the errors relative to the
$N_{\rm ref}=8192$ comparison state. The state and energy errors
decrease rapidly, with increasing empirical orders, until the
numerical floor is reached, which is consistent with spectral-type
spatial convergence for fixed regularization parameters.

For comparison, we repeat the refinement with the unmodified harmonic
potential $V_H=x^2/2$ on the periodic domain. The two calculations are
nearly indistinguishable on coarse and moderate meshes. 
At higher resolutions, the direct harmonic calculation develops an
algebraic error floor, with the observed order approaching two,
consistent with the lack of periodic smoothness of $V_H$.
This error
floor is removed by the $C^\infty$ far-field continuation, as shown
in Figure \ref{fig:mesh-convergence}.

\begin{figure}[t]
\centering
\includegraphics[width=0.55\textwidth]
{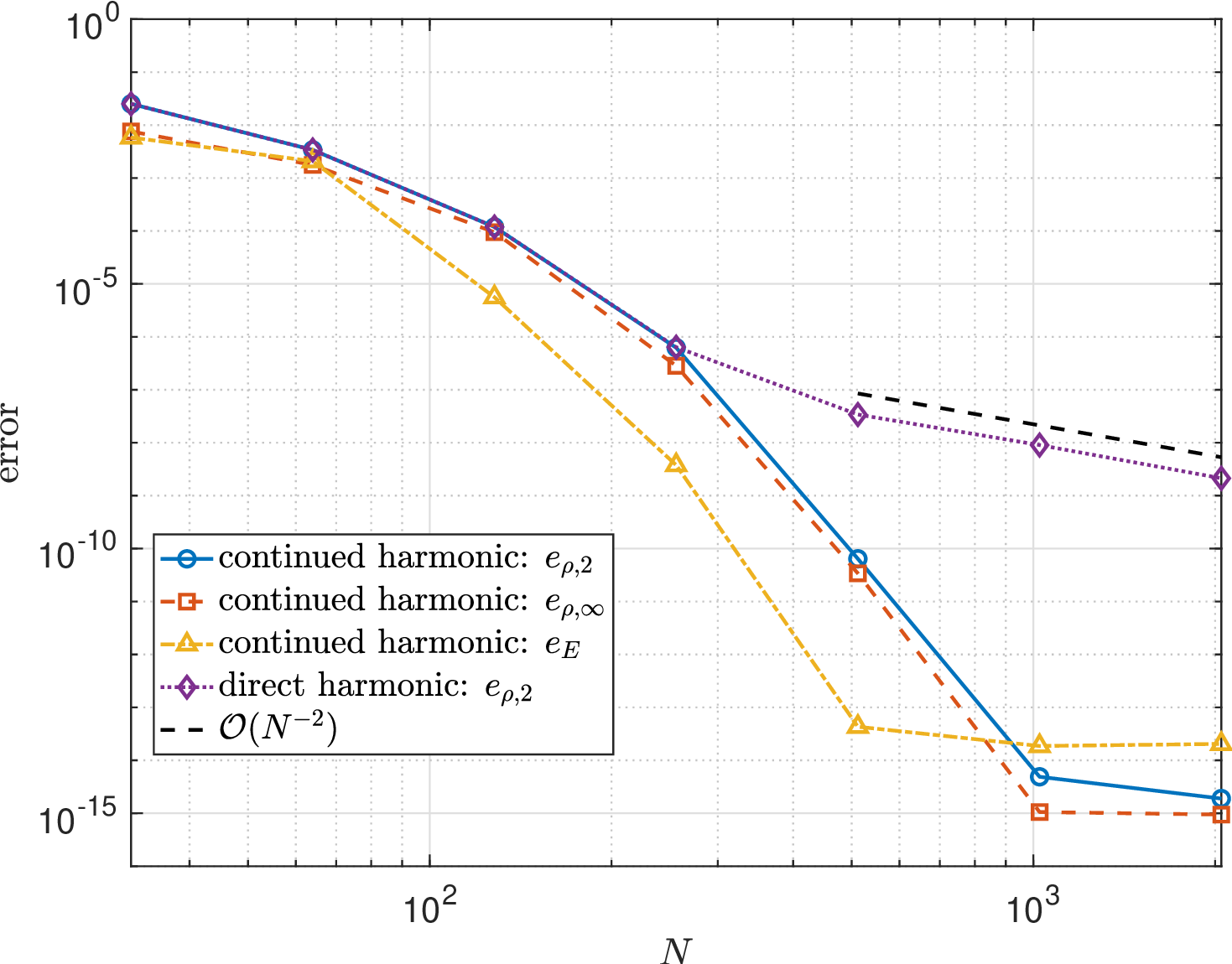}
\caption{
Fourier spatial convergence with fixed regularization parameters.
``Continued harmonic'' refers to the $C^\infty$ far-field
continuation \eqref{eq:periodic-harmonic}, and ``direct harmonic'' to
$V_H(x)=x^2/2$ on the periodic computational domain.
}
\label{fig:mesh-convergence}
\end{figure}

\subsection{Regularization effects}
\label{subsec:num-regularization}

The regularized model contains two parameters associated with the
kinetic and potential terms. We first examine the effect of the
potential regularization parameter $\sig$ at fixed $\eps$, including
its influence on Fourier resolution. We then compare different
regularizations of the density kinetic term as $\eps\to0$.

\paragraph{Potential regularization.}

We fix $\eps=10^{-2}$, $L=32$, and $N=4096$, and take
$\sig=\eps^p$ with $p=1,2,3,4$. Let $\rho_{\eps,\sig}$ denote the
minimizer with the regularized potential, and let $\rho_{\eps,0}$
denote the minimizer obtained with 
the original potential. To quantify the perturbation introduced
by $p_\sig$, we define
\begin{equation}
\label{eq:num-potential-bias}
\begin{aligned}
B_E(\sig)
&=
E_{\eps,0}(\rho_{\eps,\sig})
-
E_{\eps,0}(\rho_{\eps,0}),
\quad
B_\rho(\sig)
&=
\norm{
\rho_{\eps,\sig}-\rho_{\eps,0}
}_{L^2(\mathcal D_L)}.
\end{aligned}
\end{equation}
Here $E_{\eps,0}$ denotes the energy with the same denominator
$r_\eps$ and the original potential contribution $V\rho$.

Table~\ref{tab:sigma-sweep} shows that both perturbations decrease
approximately linearly with $\sig$ over the tested range. The minimum
density exhibits the same scaling for small $\sig$, indicating that
the positive low-density background introduced by the potential
regularization vanishes together with $\sig$.

\begin{table}[H]
\centering
\small
\caption{
Effect of the potential regularization scale at
$\eps=10^{-2}$, $L=32$, and $N=4096$.
The quantities $B_E(\sig)$ and $B_\rho(\sig)$ are defined in
\eqref{eq:num-potential-bias}.
}
\label{tab:sigma-sweep}
\begin{tabular}{cccc}
\toprule
$\sig$
&
$B_E(\sig)$
&
$B_\rho(\sig)$
&
$\min\rho$
\\
\midrule
$10^{-2}$
& $1.776\times10^{0}$
& $2.548\times10^{-2}$
& $6.449\times10^{-5}$
\\
$10^{-4}$
& $1.849\times10^{-2}$
& $2.900\times10^{-4}$
& $6.740\times10^{-7}$
\\
$10^{-6}$
& $1.849\times10^{-4}$
& $2.719\times10^{-6}$
& $6.743\times10^{-9}$
\\
$10^{-8}$
& $1.849\times10^{-6}$
& $1.993\times10^{-8}$
& $6.744\times10^{-11}$
\\
\bottomrule
\end{tabular}
\end{table}

To examine the transition introduced by the potential regularization,
we consider
\begin{equation}
\label{def:p_prime}
p_\sig'\bigl(\rho_{\eps,\sig}(x)\bigr)
=
\frac{\rho_{\eps,\sig}(x)}
{\sqrt{\rho_{\eps,\sig}(x)^2+\sig^2}},
\end{equation}
which measures the local slope of the regularized potential relative
to the original linear potential. It is close to one when
$\rho_{\eps,\sig}\gg\sig$ and close to zero when
$\rho_{\eps,\sig}\ll\sig$.

\begin{figure}[t]
\centering
\includegraphics[width=0.65\textwidth, trim = 0 30 0 0, clip]
{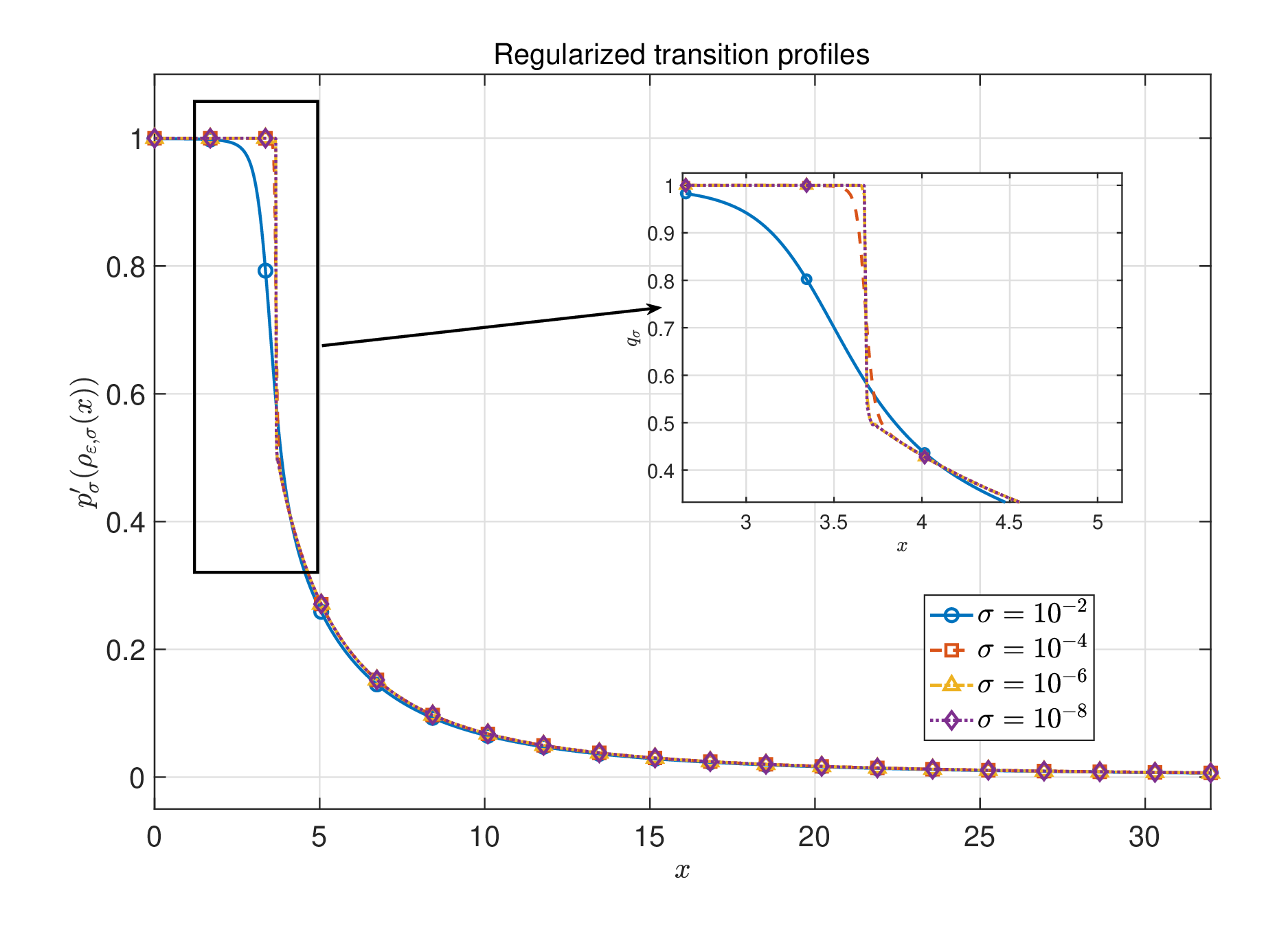}
\caption{
Profiles of
$p_\sig'(\rho_{\eps,\sig}(x))$ \eqref{def:p_prime}
for different potential-regularization scales.
The transition becomes sharper as $\sig$ decreases, and the inset
enlarges the transition region.
All computations use
$\eps=10^{-2}$, $L=32$, and $N=4096$.
}
\label{fig:regularization-sigma}
\end{figure}

Figure~\ref{fig:regularization-sigma} shows that the transition
between these two regimes becomes sharper as $\sig$ decreases.
Accordingly, the modification introduced by the potential
regularization becomes more localized near the low-density region.
The resulting smaller spatial scale requires more Fourier modes to
resolve accurately.


To examine this effect on Fourier resolution, we repeat the mesh
refinement for
$\sig=0$, $10^{-2}$, $10^{-4}$, and $10^{-6}$, while keeping
$\eps=10^{-2}$ and $L=32$. Here $\sig=0$ corresponds to the original
linear potential without potential regularization. We use the same
$C^\infty$-continued harmonic potential
\eqref{eq:periodic-harmonic} and the $N_{\rm ref}=65536$ solution as
the reference.

Figure~\ref{fig:sigma-accuracy} shows the resulting energy and state
errors. For $\sig=10^{-2}$, both errors rapidly enter the
spectral-type regime and reach their numerical floors. As $\sig$
decreases, a longer pre-asymptotic range appears before the rapid
decay becomes visible. 
In contrast, for $\sig=0$, the nonzero slope of the linear potential
at zero density can activate the positivity constraint and reduce the
regularity of the density. Consistently, spectral-type decay is not
observed over the tested resolutions, and the errors decrease much
more slowly.
This reveals a bias--resolution tradeoff in which decreasing $\sig$
reduces the regularization error but requires a finer Fourier mesh to
maintain the same spatial accuracy.

\begin{figure}[H]
\centering
\includegraphics[width=0.45\textwidth]
{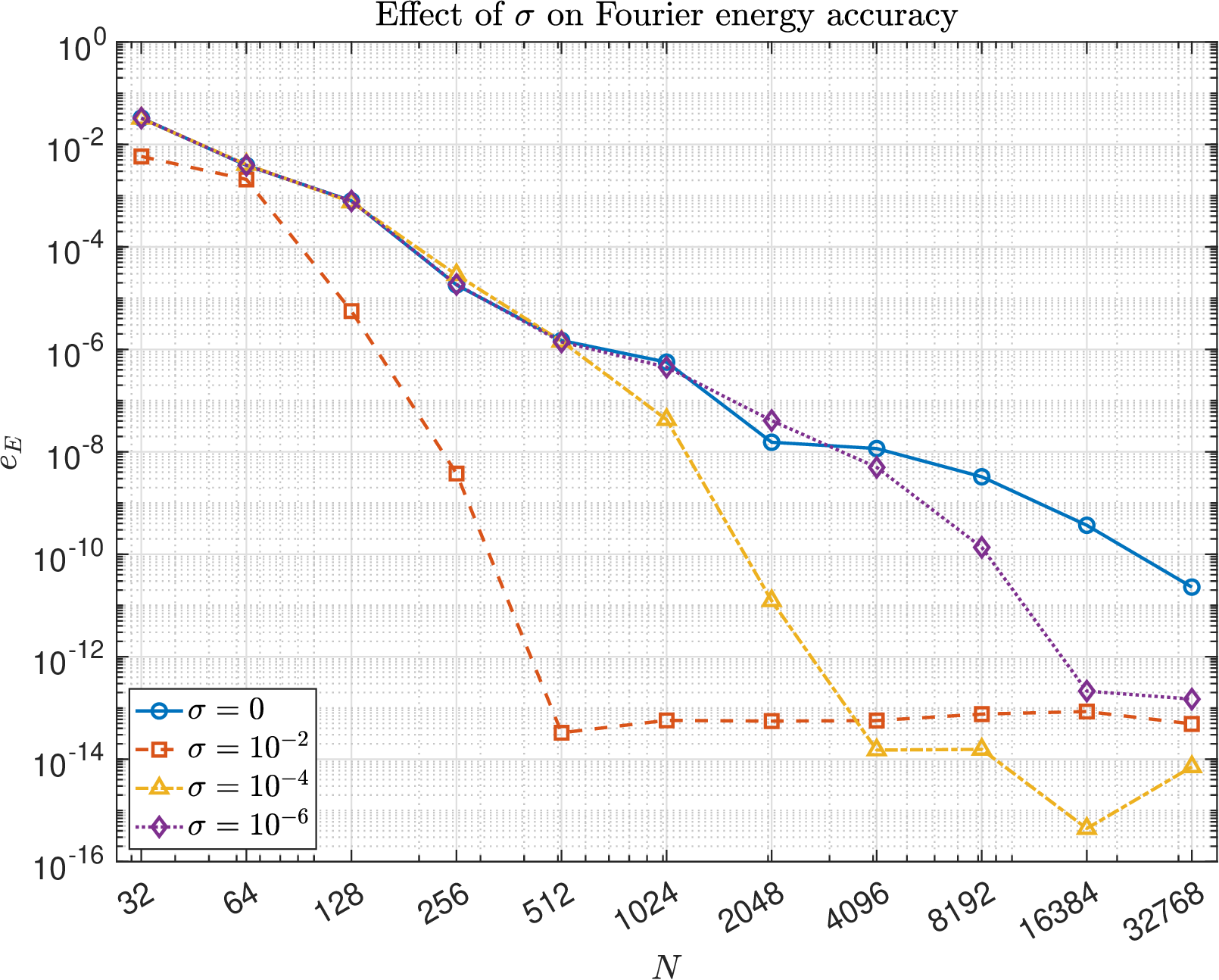}
\includegraphics[width=0.45\textwidth]
{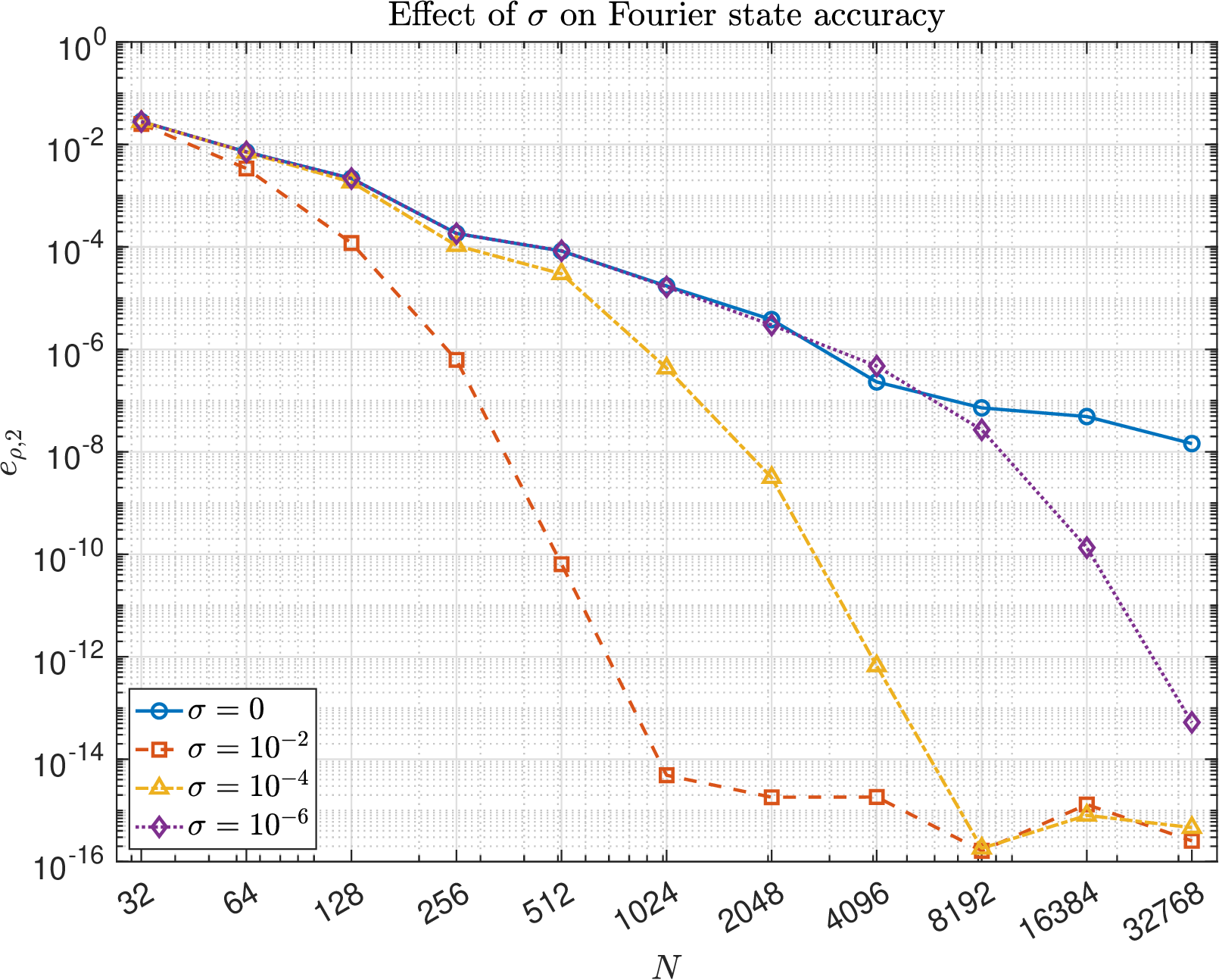}
\caption{
Effect of the potential regularization scale on Fourier accuracy.
The left panel shows the energy error and the right panel shows the
 $L^2$ state error. Each curve refines a fixed-$\sig$
problem with $\eps=10^{-2}$, $L=32$, and the harmonic-core
$C^\infty$ far-field continuation, relative to the
$N_{\rm ref}=65536$ comparison state.
}
\label{fig:sigma-accuracy}
\end{figure}

\paragraph{Regularization of the kinetic term.}

We next show the effect of the kinetic regularization by fixing the
potential regularization and the spatial discretization. We compare
the shifted denominator \eqref{eq:shift-denominator} with the
piecewise $C^1$ and $C^2$ denominators in
\eqref{eq:piecewise-regularization-family}. Throughout this test, we
take $\sig=10^{-2}$, $L=32$, and $N=512$, and use the same
$C^\infty$-continued harmonic potential. The tested values are
$\eps=10^{-1},5\times10^{-2},2\times10^{-2},10^{-2},
5\times10^{-3},2\times10^{-3},10^{-3}$.

A common reference state $\rho_{\rm ref}$ is computed with the shifted
denominator at $\eps_{\rm ref}=10^{-5}$. To compare the three
regularizations on the same basis, we measure the state error and
evaluate the energy error with the common unregularized kinetic
denominator $r_0(\rho)=\rho$. 
We define
\begin{equation}
\label{eq:kinetic-reg-errors}
\begin{aligned}
e_\rho(\eps)
&=
\norm{\rho_\eps-\rho_{\rm ref}}_{L^2(\mathcal D_L)},
\quad
e_E(\eps)
&=
\abs{
E_{0,\sig}(\rho_\eps)
-
E_{0,\sig}(\rho_{\rm ref})
}.
\end{aligned}
\end{equation}
Here $E_{0,\sig}$ denotes the energy with the original kinetic
denominator $r_0(s)=s$ and the fixed potential regularization
$p_\sig$. 

\begin{figure}[H]
\centering
\includegraphics[width=0.48\textwidth]
{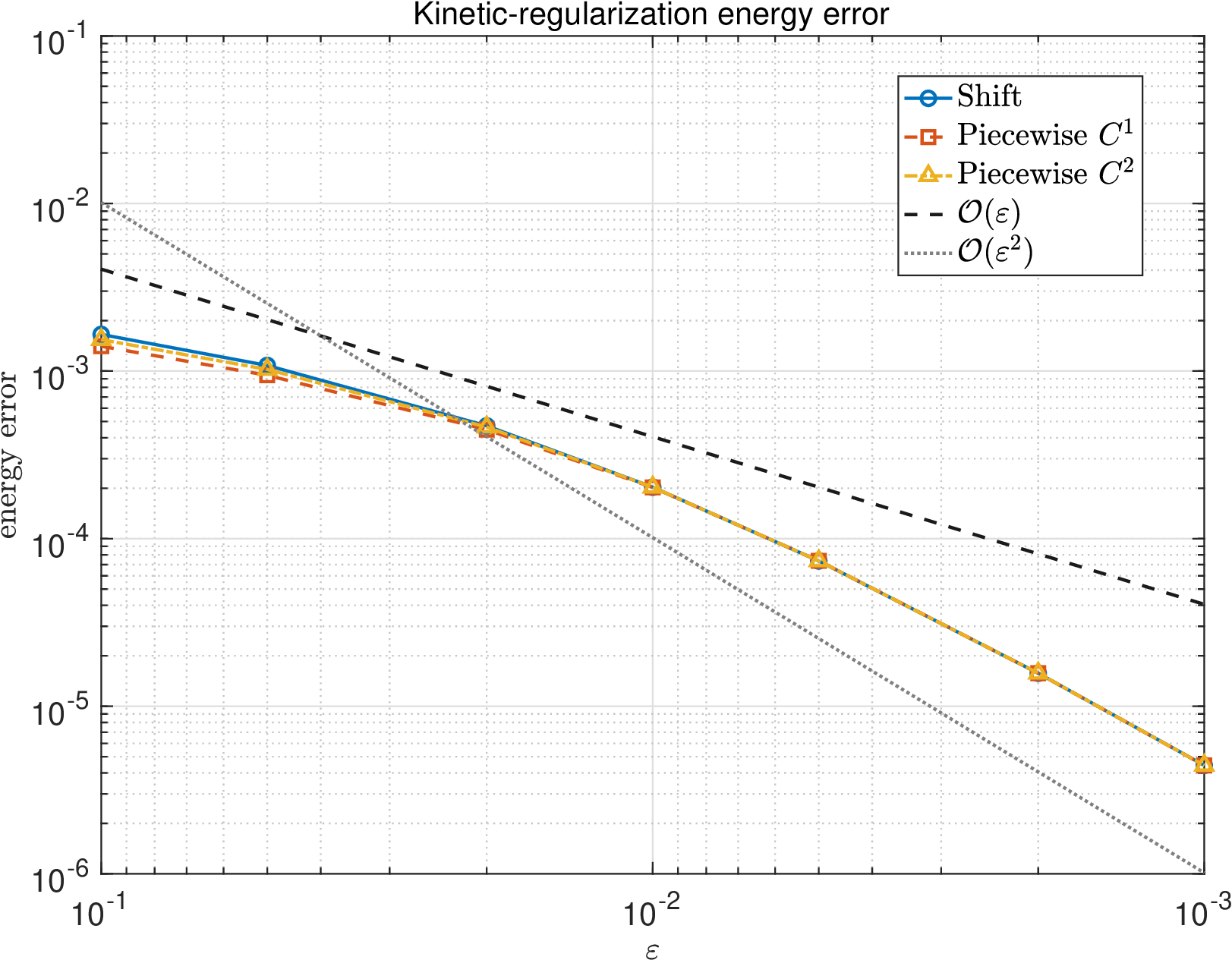}
\hfill
\includegraphics[width=0.48\textwidth]
{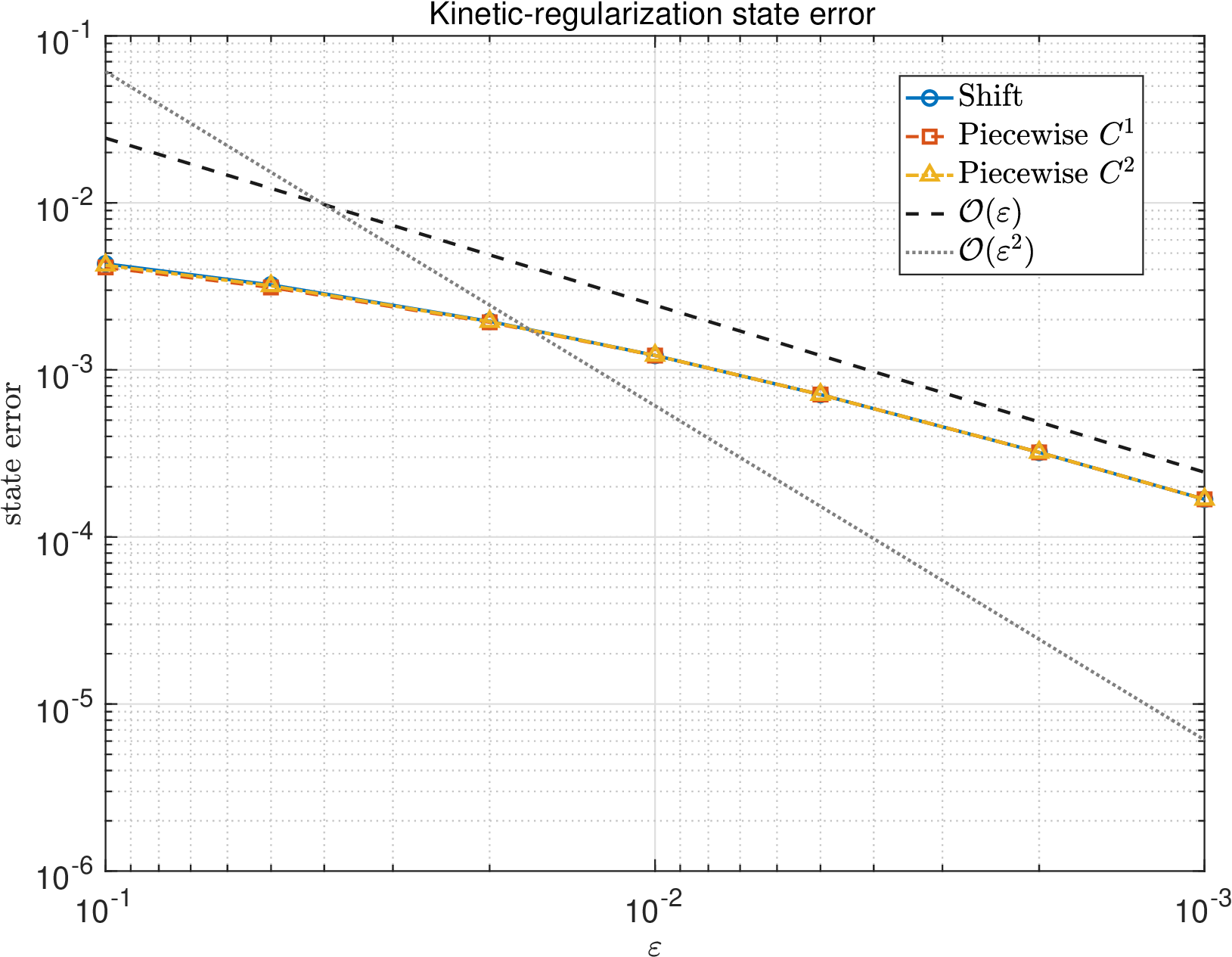}
\caption{
Effect of the kinetic regularization. The shifted denominator \eqref{eq:shift-denominator} and the
piecewise $C^1$ and $C^2$ denominators \eqref{eq:piecewise-regularization-family} are compared using the errors
defined in \eqref{eq:kinetic-reg-errors}. The left panel shows the
energy error  \(e_E(\eps)\) and the right panel the $L^2$ state error \(e_\rho(\eps)\). The dashed and
dotted lines indicate first- and second-order slopes, respectively.
All computations use $\sig=10^{-2}$, $L=32$, and $N=512$.
}
\label{fig:kinetic-regularization}
\end{figure}

Figure~\ref{fig:kinetic-regularization} shows that the three
regularizations have essentially the same asymptotic behavior as
$\eps$ decreases. 
The common-functional energy error approaches second order whereas the state error is approximately first order.
The piecewise $C^1$ denominator gives a modest reduction in error for
the larger values of $\eps$, but neither piecewise regularization
shows a clear asymptotic rate advantage over the shifted denominator.
This supports using the simpler shifted denominator in the solver
experiments.

\subsection{Efficiency tests}
\label{subsec:num-solver}

We next examine the computational performance of the two-stage
optimization strategy. Throughout this subsection, we use the shifted
denominator $r_\eps(\rho)=\rho+\eps$. We first examine the
respective roles of FISTA and the Newton refinement, and then
examine how the two regularization scales affect the computational
cost.

\paragraph{FISTA and Newton refinement.}

We consider the representative case
$\eps=10^{-2}$, $\sig=10^{-4}$, $L=32$, and $N=4096$.
All runs start from the projection of $e^{-x^2}/\sqrt{\pi}$ onto
$\CN$. 
After the first $200$ FISTA iterations, we switch to the Newton
refinement when the relative energy variation
$(E_{\max}-E_{\min})/\max\{1,|E_k|\}$ over the preceding $50$
iterations is below $10^{-12}$ and
$\norm{\mathcal G_1(\rho^k)}_h\le10^{-5}$.
Here $E_{\max}$ and $E_{\min}$ denote the maximum and minimum energy
over these $50$ iterations, and $E_k$ is the current energy.
Both conditions must hold for two consecutive iterations. 
Starting from
the same switching state, one branch continues with FISTA, while the other
proceeds with the Newton refinement. The two potential-proximal runs
therefore share the same FISTA history up to the switch.

For comparison, we also consider a smooth-potential FISTA, in which
the regularized potential is included in the smooth part instead of the proximal part. 
It minimizes the same
discrete energy and uses the same FISTA backtracking rule.

Since FISTA and Newton steps have different computational costs,
Figure~\ref{fig:solver-performance} reports the convergence histories
against cumulative wall-clock time. 
The two branches share the same FISTA time history up to the switching
point. Thereafter, the additional time of each branch is measured from
the common switching time.
The left panel shows the energy error
relative to the final two-stage solution, while the right panel shows
the projected-gradient residual $\norm{\mathcal G_1(\rho^k)}_h$. 
As shown in Figure~\ref{fig:solver-performance}, FISTA rapidly reduces
the energy and produces an accurate approximation of the ground state,
but continued FISTA decreases the projected-gradient residual only
slowly. Starting from the same switching state, the Newton refinement
reduces the residual much more rapidly.

\begin{figure}[H]
\centering
\includegraphics[width=0.48\textwidth]
{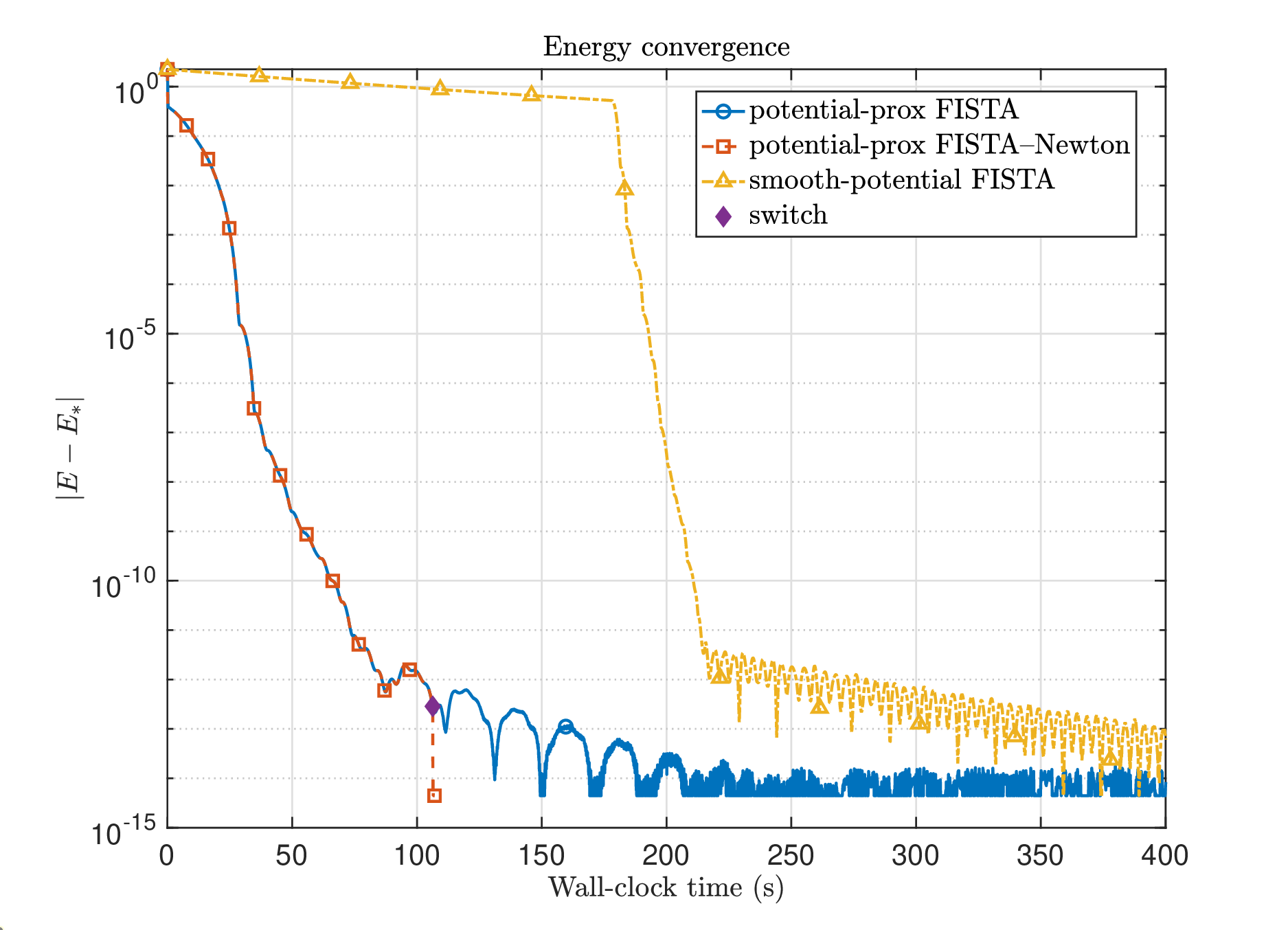}
\hfill
\includegraphics[width=0.48\textwidth]
{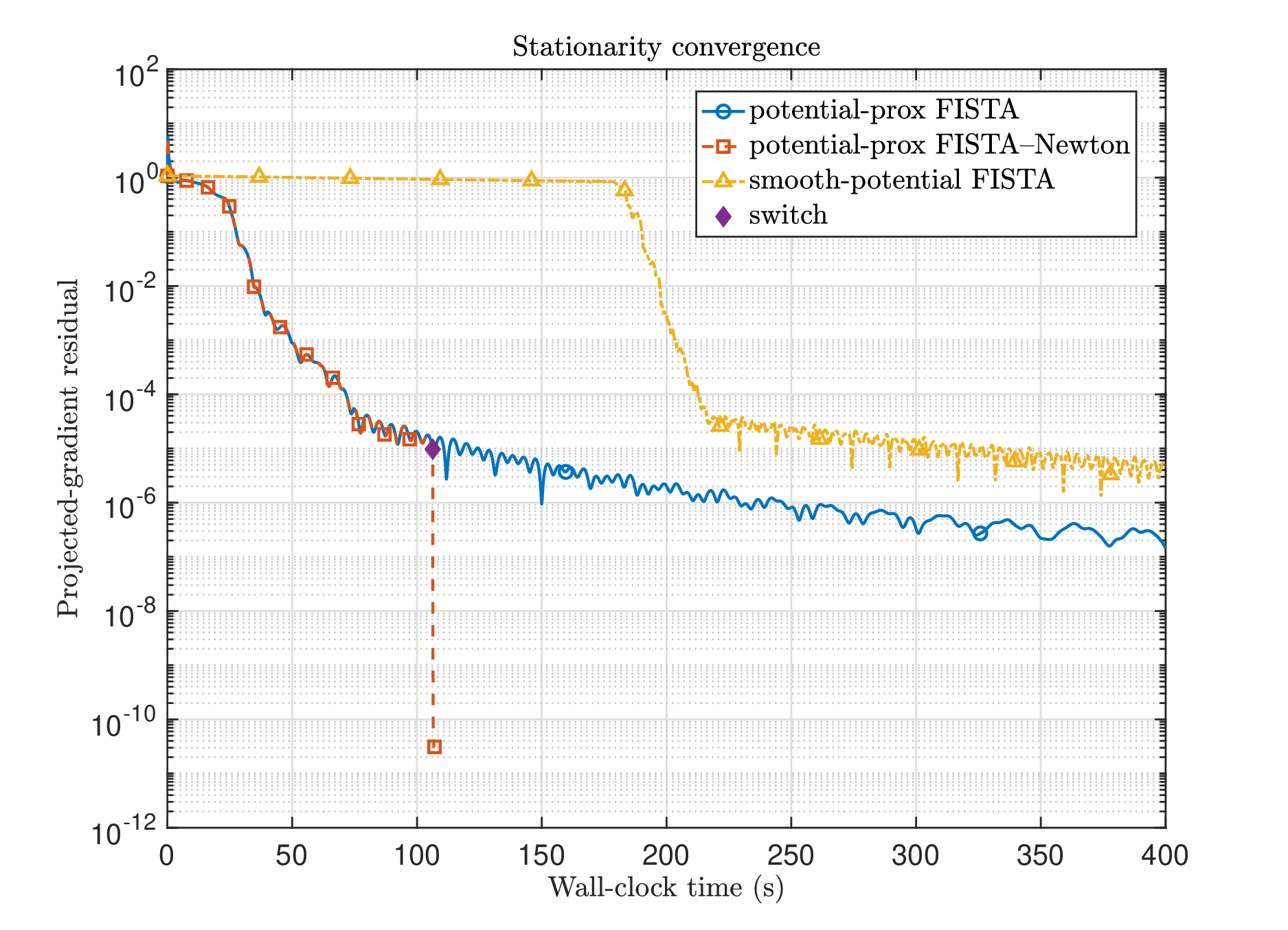}
\caption{
Convergence histories for continued FISTA and the two-stage
FISTA--Newton method with
$\eps=10^{-2}$, $\sig=10^{-4}$, $L=32$, and $N=4096$.
The left panel shows the energy error relative to the final two-stage
solution, and the right panel shows the projected-gradient residual
$\norm{\mathcal G_1(\rho^k)}_h$.
Here ``potential-proximal FISTA'' and ``smooth-potential FISTA''
refer to placing the regularized potential in the proximal and smooth
parts of the splitting, respectively.
}
\label{fig:solver-performance}
\end{figure}

\paragraph{Dependence on the regularization scales.}

We next examine how the computational cost changes as the
regularization parameters decrease. To isolate the two effects, we
first vary $\eps$ at fixed $\sig=10^{-2}$ and then vary $\sig$ at
fixed $\eps=10^{-2}$. In all cases,
$r_\eps(\rho)=\rho+\eps$.

All runs use the same Fourier resolution $N=4096$, initial density,
and solver parameters, and the results are summarized in
Table~\ref{tab:solver-regularization-cost}.
All computations were performed in MATLAB R2023a on a laptop
with an Apple M4 chip.
The reported wall-clock time includes both the FISTA and Newton stages.
The FISTA stage is limited to $20000$ iterations.
If the standard switching criterion has not been met by then, the
iterate is passed to the Newton stage provided that
$\|\mathcal G_1(\rho)\|_h\le10^{-4}$.
The asterisk in Table~\ref{tab:solver-regularization-cost} marks the
run that reached this limit.


\begin{table}[H]
\centering
\caption{
Computational cost of the two-stage solver for different
regularization scales with $r_\eps(\rho)=\rho+\eps$ and $N=4096$.
``FISTA'' and ``Newton'' denote the iteration counts in the two
stages, and the last column reports the final projected-gradient
residual $\|\mathcal G_1(\rho)\|_h$ \eqref{eq:pg-map}.
The asterisk marks a run that reached the FISTA iteration limit.
}
\label{tab:solver-regularization-cost}
\begin{tabular}{cccccc}
\toprule
$\eps$ & $\sig$
& FISTA & Newton
& time (s) & $\|\mathcal G_1(\rho)\|_h$ \\
\midrule
$10^{-2}$ & $10^{-2}$ & 3422 & 8
& 93.50 & $5.04\times10^{-11}$ \\
$10^{-3}$ & $10^{-2}$ & 9042 & 10
& 201.64 & $2.54\times10^{-10}$ \\
$10^{-4}$ & $10^{-2}$ & $20000^*$ & 14
& 519.77 & $1.15\times10^{-9}$ \\
\midrule
$10^{-2}$ & $10^{-3}$ & 2992 & 10
& 99.28 & $3.93\times10^{-11}$ \\
$10^{-2}$ & $10^{-4}$ & 3025 & 18
& 107.02 & $3.11\times10^{-11}$ \\
\bottomrule
\end{tabular}
\end{table}

Table~\ref{tab:solver-regularization-cost} shows that decreasing
$\eps$ at fixed $\sig=10^{-2}$ substantially increases the FISTA
iteration count and the total solution time. The smallest $\eps$
case reaches the prescribed FISTA iteration limit and has a somewhat
larger final residual. This increase is consistent with the larger
coefficients involving $1/r_\eps(\rho)$ in the kinetic Hessian,
particularly in low-density regions.

In contrast, decreasing $\sig$ at fixed $\eps=10^{-2}$ produces only
a modest change in the computational cost over the tested range.
Although a smaller $\sig$ sharpens the potential transition, the
potential-proximal splitting keeps the large curvature
$p_\sig''(0)=1/\sig$ out of the smooth FISTA step.

\subsection{Two-dimensional examples}
\label{subsec:num-2d}

We conclude with two examples in two dimensions. In both cases, we
use the shifted denominator $r_\eps(\rho)=\rho+\eps$, the potential
regularization $p_\sig$, and
\[
\beta=\delta=10,\qquad
\eps=10^{-2},\qquad
\sig=10^{-4}.
\]
The computations are performed on the periodic domain
$[-16,16)^2$ with $N_x=N_y=512$ Fourier collocation points. The
continuation \eqref{eq:periodic-harmonic} is applied separately in
the $x$- and $y$-directions. Thus the numerical harmonic potential is
$\gamma_x^2V_L^{\rm p}(x)+\gamma_y^2V_L^{\rm p}(y)$ and agrees with
the stated trap on $|x|,|y|\le0.75L$.

We first consider the anisotropic harmonic trap
\begin{equation}
\label{eq:2d-anisotropic-harmonic}
V_{\rm ah}(x,y)
=
\frac12
\left(
\gamma_x^2x^2+\gamma_y^2y^2
\right),
\qquad
\gamma_x=1,
\quad
\gamma_y=2.
\end{equation}
The resulting ground-state density is shown in the left panel of
Figure \ref{fig:two-dimensional-examples}. The stronger confinement in
the $y$-direction gives a density profile that is narrower in that
direction.

For the second example, we add a square optical lattice,
\begin{equation}
\label{eq:2d-optical-lattice}
V_{\rm ol}(x,y)
=
V_{\rm ah}(x,y)
+
V_0
\left[
\sin^2(kx)+\sin^2(ky)
\right],
\end{equation}
with $V_0=5$ and $k=\pi/2$. The resulting density is shown in the
right panel of Figure \ref{fig:two-dimensional-examples}. The optical
lattice introduces multiple local wells and produces the corresponding
modulation of the ground-state density.

\begin{figure}[htbp]
\centering
\includegraphics[width=0.48\textwidth]
{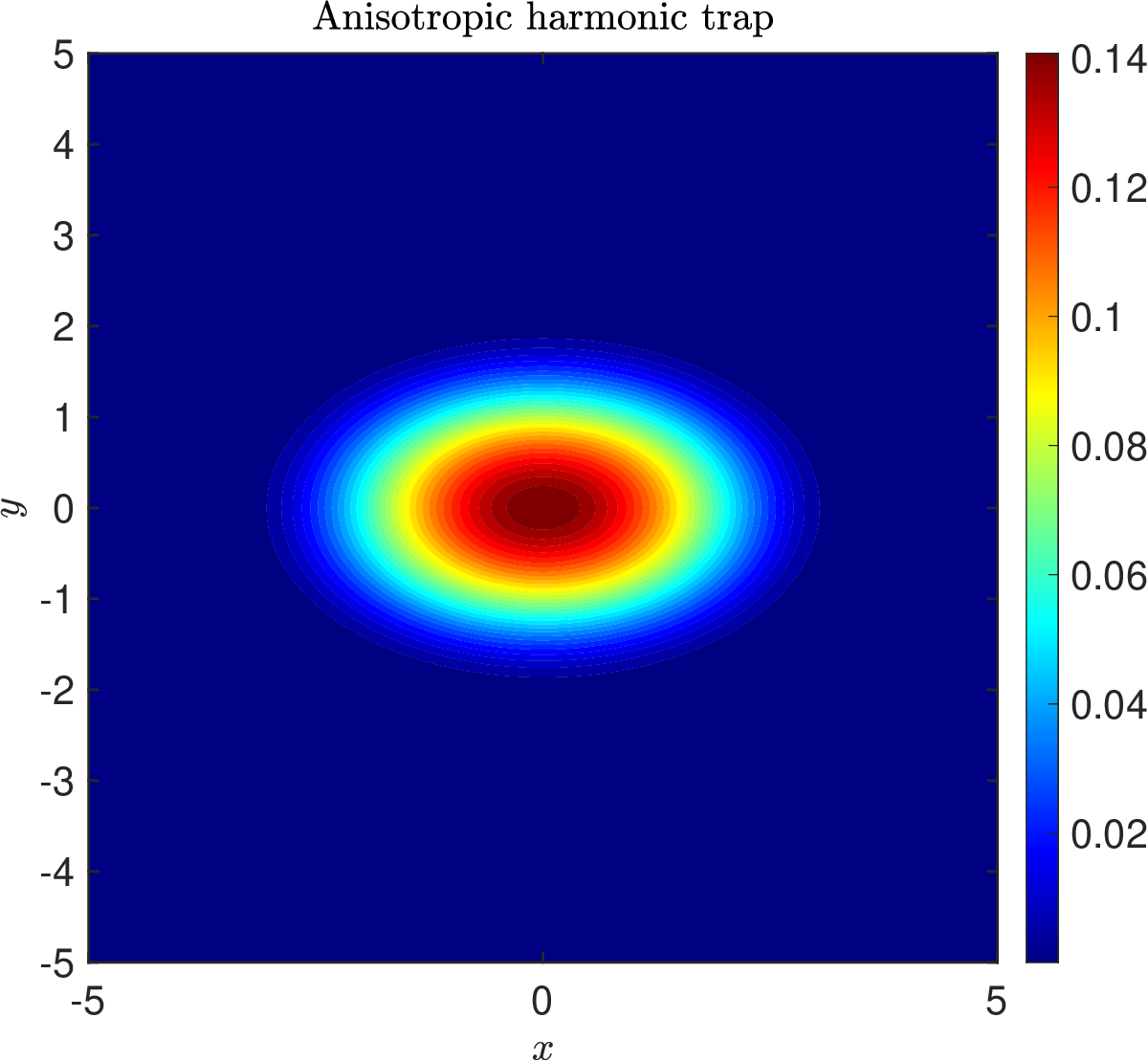}
\hfill
\includegraphics[width=0.48\textwidth]
{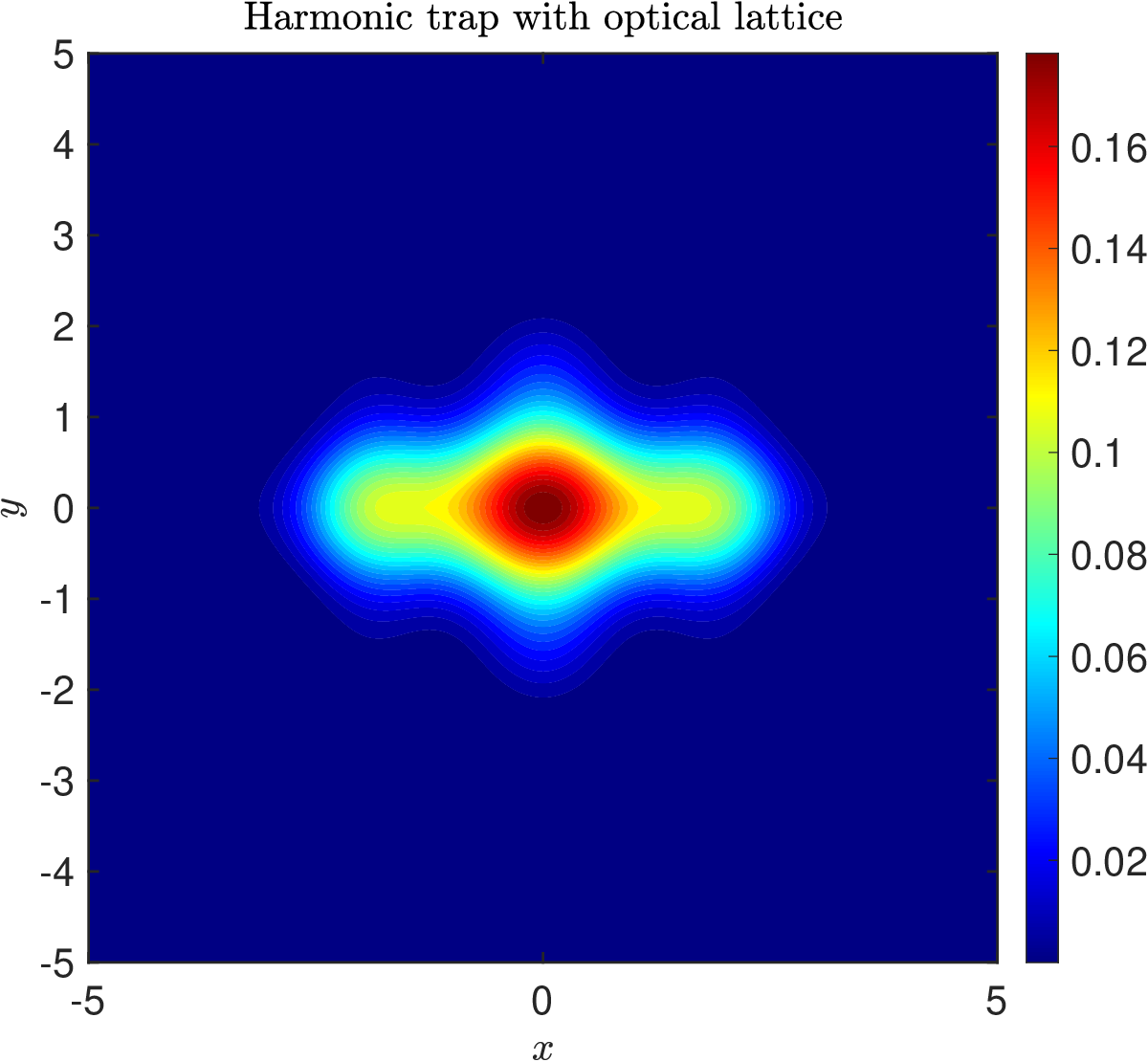}
\caption{
Two-dimensional ground-state densities computed on $[-16,16)^2$
with $N_x=N_y=512$, with the central region shown.
The left panel corresponds to the anisotropic harmonic trap
\eqref{eq:2d-anisotropic-harmonic}, while the right panel includes
the square optical lattice \eqref{eq:2d-optical-lattice}.
Both computations use $\eps=10^{-2}$, $\sig=10^{-4}$, and
$\beta=\delta=10$.
}
\label{fig:two-dimensional-examples}
\end{figure}

\section{Conclusion}
\label{sec:summary}

%

We have developed a positive-conservative Fourier optimization method
for computing ground states of Bose--Einstein condensates with
higher-order interactions in the density formulation. The kinetic and
potential terms are regularized while preserving the convex structure
of the density energy. On the computational domain, the regularized
energies $\Gamma$-converge to the original density energy as the
regularization parameters vanish.

The numerical results show spectral-type convergence of the Fourier
discretization when the regularization parameters are fixed. 
They also show
that decreasing the potential regularization parameter reduces the
regularization error but requires a finer Fourier resolution to resolve
the sharper low-density transition. In the two-stage optimization,
FISTA provides the main energy reduction, while the subsequent Newton
refinement further reduces the stationarity residual to high accuracy.

A joint analysis of domain truncation, Fourier discretization, and
regularization remains for future work.


\bigskip
\noindent{\bf Acknowledgements} 
This work was partially supported by the National Natural Science
Foundation of China under Grants 12201436 (X. Ruan) and 12501547 (B. Lin).

\appendix

%
%

\section{Implementation of the Newton refinement}
\label{app:newton-implementation}

For the Newton refinement, let
\[
a=r_\eps(\rho),\qquad
a'=r_\eps'(\rho),\qquad
a''=r_\eps''(\rho),\qquad
q=D_N\rho,\qquad
\dot q=D_N\eta.
\]
The products and quotients in the following formula are understood
componentwise. The Hessian action of the regularized kinetic term is
\begin{align}
\label{eq:hessian-action}
H_{\rm kin}(\rho)[\eta]
={}&
\frac14D_N^T
\left(
\frac{\dot q}{a}
-
\frac{q\,a'}{a^2}\eta
\right)
-
\frac18
\left[
2q\dot q\,\frac{a'}{a^2}
+
q^2
\left(
\frac{a''}{a^2}
-
\frac{2(a')^2}{a^3}
\right)\eta
\right].
\end{align}
The full Hessian action is
\begin{equation}
\label{eq:full-hessian-action}
H(\rho)[\eta]
=
H_{\rm kin}(\rho)[\eta]
+
Vp_\sig''(\rho)\eta
+
\beta\eta
+
\delta D_N^TD_N\eta.
\end{equation}

For the PCG preconditioner, let $D_{\rm FD}$ denote the periodic
forward-difference operator and define
\[
B_{\rm FD}
=
D_{\rm FD}
-
\operatorname{diag}\left(\frac{q\,a'}{a}\right),
\qquad
W
=
\operatorname{diag}\left(\frac{1}{4a}\right),
\qquad
R_{\rm kin}
=
\operatorname{diag}\left(
-\frac{q^2a''}{8a^2}
\right).
\]
We use the sparse matrix
\begin{equation}
\label{eq:fd-preconditioner}
P_{\rm FD}
=
B_{\rm FD}^TWB_{\rm FD}
+
R_{\rm kin}
+
\delta D_{\rm FD}^TD_{\rm FD}
+
\operatorname{diag}\bigl(\beta+Vp_\sig''(\rho)\bigr)
\end{equation}
as the PCG preconditioner. In the strongly convex regime,
$P_{\rm FD}$ is symmetric positive definite.

\end{document}